\documentclass[11 pt,reqno]{article}
\usepackage[latin1]{inputenc}
\usepackage{natbib}
\usepackage{amsmath,amsfonts,amssymb,amsthm}
\usepackage{mathrsfs}
\usepackage{latexsym}
\usepackage[english]{babel}
\usepackage{multirow}

\usepackage{graphicx}
\usepackage[usenames,dvipsnames]{xcolor}

\usepackage{float}
\usepackage{color}

\usepackage{amssymb,amsmath,color,graphicx}

\usepackage[bookmarks = true, colorlinks=true, linkcolor = blue, citecolor = blue, menucolor = black, urlcolor = black, plainpages=false]{hyperref}
\usepackage[left=2cm,top=2.5cm,right=2cm,nohead,foot=1cm]{geometry}

\newcommand{\rmd}{{\mathrm{d}}}
\newcommand{\R}{\mathbb{R}}

\newcommand{\E}{\operatorname{\mathbb{E}}}
\newcommand{\Var}{\operatorname{Var}}

\newcommand{\rme}{{\rm e}}

\newtheorem{Theorem}{Theorem}[section]
\newtheorem{Corollary}{Corollary}[section]
\newtheorem{Lemma}{Lemma}[section]
\newtheorem{Proposition}{Proposition}[section]
\newtheorem{Definition}{Definition}[section]
\newtheorem{Remark}{Remark}

\usepackage{ifpdf}
\ifpdf
\else
\usepackage{graphicx}
\DeclareGraphicsExtensions{.eps,.jpg,.png}
\fi
\usepackage{epstopdf}
\begin{document}

\title{\bf{On the excursion area of perturbed Gaussian fields}}

\author{\sc Elena Di Bernardino\footnote{Conservatoire National des Arts et M\'{e}tiers, Paris, EA4629,  292 rue Saint-Martin, Paris Cedex 03, France; elena.dibernardino@lecnam.net}, Anne Estrade\footnote{MAP5 UMR CNRS 8145, Universit\'e Paris Descartes,  45 rue des Saints-P\`eres, 75270 Paris Cedex 06, France; anne.estrade@parisdescartes.fr}\, and Maurizia Rossi\footnote{Dipartimento di Matematica, Universit\`a di Pisa, 5 Largo Bruno Pontecorvo, 56127 Pisa, Italia; maurizia.rossi@unipi.it}}

 \maketitle

 \begin{abstract}
\noindent We investigate Lipschitz-Killing curvatures for excursion sets of random fields on $\mathbb R^2$ under small spatial-invariant random perturbations. An expansion formula for mean curvatures is derived when the magnitude of the perturbation vanishes, which recovers the Gaussian Kinematic Formula at the limit by contiguity of the model. We develop an asymptotic study of the perturbed excursion area behaviour that leads to a quantitative non-Gaussian limit theorem, in Wasserstein distance, for fixed small perturbations and growing domain. When letting both the perturbation vanish and the domain grow, a standard Central Limit Theorem follows. Taking advantage of these results, we propose an estimator for the perturbation which turns out to be asymptotically normal and unbiased, allowing to make inference through sparse information on the field.
\smallskip


\smallskip


 \end{abstract}


\section{Introduction}

A wide range of phenomena can be seen as single realizations of a random field, for instance the Cosmic Microwave Background radiation (CMB) (see \citet{MPbook}), medical images of brain activity (see \citet{Worsley}) and of mammary tissue (see \citet{Burgess}) and many others.  Their features can be investigated through geometrical functionals, among them the well-known class of Lipschitz-Killing (LK) curvatures of excursion sets (see e.g. \citet{Weil08} and \citet{Thale08} for a precise definition and \cite{Yabi2} for some applications in cosmology).
 From a theoretical point of view,  probabilistic and statistical properties of the latter have been widely studied in the last decades.  For instance, in the two-dimensional Euclidean setting, in \citet{cabana1987affine, BD16, Berzin}, the length of the level sets (\emph{i.e.} the perimeter of the excursion sets) is taken into account, in \citet{annejose} the Euler-Poincar\'e characteristic, while several limit theorems are obtained for the excursion area in \citet{bulinski2012central, Spodarev13}. See \citet{KV16, Mu16} for higher dimensions. In this manuscript we focus on the two-dimensional setting, \emph{i.e.} random fields defined on $\mathbb R^2$ endowed with the standard Euclidean metric.

In many cases, the LK curvatures are studied for Gaussian excursion sets via the Gaussian Kinematic Formula  (see, \emph{e.g.}, \citet{AT07, BDDE}). In this framework, a natural question is the following: \emph{how do these geometric quantities change under small perturbations of the underlying field?} The present work gives an answer in the case of an independent, additive, spatial-invariant perturbation of a stationary isotropic Gaussian field. Indeed, this model naturally arises when taking into account measurement errors that globally affect all the observations in a physical experiments.
As briefly anticipated above, LK curvatures have been very extensively exploited in the recent cosmological
literature as a tool to probe non-Gaussanity and anisotropies in the
CMB (see \emph{e.g.} \cite{Yabi1} and \cite{Yabi2}). Our setting could be viewed as the
representation of a Gaussian field contaminated by super-imposed point
sources (i.e., galaxies and other astrophysical objects), and in this
sense it could be used for point source detection or map validation in
the framework of CMB data analysis.
We remark that the perturbation of a Gaussian field obtained by adding either an independent Gaussian field or a function of the field itself can be fully treated  through Gaussian techniques (see  \citet{BTV}). Some computations of expected values of LK curvatures of
excursion sets in
the latter setting have been given in the physical
literature by e.g. \cite{Matsubara1} and \cite{Matsubara2} in order to derive a promising
method to constrain the primordial non-Gaussianity of the universe by temperature fluctuations
in the CMB.

However,
our aim is not to develop the theory of LK curvatures for excursion sets of general non-Gaussian fields (see for instance  \citet{adler_samorodnitsky_taylor_2010, BD16, Rapha})  but to go beyond Gaussianity by introducing a small perturbation of the underlying Gaussian field. This perturbation clearly appears in the LK curvatures of excursion sets, allowing us to measure the discrepancy between the original and perturbed fields. Moreover, we are able to recover the classical Gaussian case by contiguity, \emph{i.e.}, when the perturbation vanishes.

Our model can be seen as a random affine transformation  of the initial excursion level. A deterministic and more challenging counterpart has been recently studied in \citet{nodal} where a different geometrical functional is considered, namely, the number of connected components of excursion sets.

At last, considering excursion sets instead of the whole field is a sparse information that is commonly used by practitioners (see, for instance, Chapter 5 in \citet{AT11}). Furthermore, it is equivalent to consider thresholded fields, which is a standard model in physics literature  (see, \emph{e.g.}, \citet{ImageAnalStereol751, Roberts1995, Roberts1999}).


\medskip

 {\it Main contributions.} In this paper, we provide an expansion formula for the perturbed LK curvatures (see Proposition \ref{LKadditive}) where the contiguity property clearly appears for a vanishing perturbation. Visually, the perturbation is not evident by looking at the image of excursion sets (see Figure 1) but its impact can be detected by an image processing through the evaluation of their LK curvatures (see Figure 3). Moreover, an asymptotically normal and unbiased estimator for the variance of the perturbation magnitude is proposed in Proposition \ref{epsilonProposition}. In order to get the Gaussian limiting behaviour of the latter estimator, we develop an asymptotic study of the second LK curvature, \emph{i.e.} the area, of the perturbed excursion sets. We analyze both the case when the perturbation vanishes and the domain of observation grows to $\mathbb R^2$ (see Theorem \ref{theoremThetaAsym}) and the case of a fixed small perturbation and growing domain (see Theorem \ref{quantitative}). The former is a standard CLT result, the latter is a quantitative limit theorem towards a non-Gaussian distribution, giving an upper bound for the convergence rate in Wasserstein distance. We deeply study the unusual non-Gaussian limiting law (see Theorem \ref{theoTCLZ} and Figures \ref{fdeltafigure}, \ref{ZepsilonFigure} and \ref{ZepsilonFigure2}).

 An auxiliary result which is of some interest for its own is collected in Lemma  \ref{lem_app}  where uniform rates  (w.r.t. the level) of convergence for sojourn times  of general Gaussian fields are proved.  An argument similar to the one in the proof of Lemma \ref{lem_app} allows to obtain uniform rates of convergence also for sojourn times of random hyperspherical harmonics, at the cost of getting worse rates than those found in  \citet{maudom}.

\medskip

{\it Outline of the paper.} Section \ref{LK} is devoted to the study of mean LK curvatures of our perturbed model. In particular, in Section \ref{preliminary}  we recall the notion of LK curvatures for Borel sets, and then introduce our setting; in Section \ref{LKmean} we derive the asymptotic expansion for the mean curvatures as the perturbation vanishes (Proposition \ref{LKadditive}) providing some numerical evidence in Figures \ref{figureLKGAUSSIAN} and  \ref{figureLKperturbed}.

In Section \ref{asymtoticSmallEpsilon} we state and prove the quantitative limit theorem, in Wasserstein distance, for the excursion area of the perturbed model for fixed small perturbations and growing domain (Theorem \ref{quantitative}) .  Theorem \ref{theoTCLZ} characterizes the unusual non-Gaussian limiting distribution whose numerical investigation leads to Figure \ref{ZepsilonFigure} and Figure \ref{ZepsilonFigure2}. In Section \ref{CLTarea} we state and prove the standard CLT for the excursion area for growing domain and disappearing perturbation (Theorem \ref{theoremThetaAsym}).

Taking advantage of the asymptotic studies for LK curvatures in Section \ref{LK} and Section \ref{asymtoticSmallEpsilon}, in Proposition \ref{epsilonProposition} we prove that the proposed estimator for the perturbation variance is unbiased and asymptotically normal. Its performance can be appreciated in Figure \ref{varepsilonFigure1}.

Finally, Appendix \ref{appS} collects the auxiliary result on uniform rates of convergence for sojourn times of Gaussian fields.

\section{LK curvatures  for the considered  perturbed  Gaussian  model}\label{LK}

\subsection{Definitions and preliminary notions}\label{preliminary}

In the present paper we consider  the   three additive functionals, called in the literature   intrinsic volumes, Minkowski functionals or Lipschitz-Killing curvatures,  $L_j$ for $j=0,1,2$, defined on subsets of Borelians in $\R^2$.  Roughly speaking, for $A$ a Borelian set in $\R^2$, $L_0(A)$ stands for the Euler characteristic of $A$, $L_1(A)$ for the half perimeter of its boundary and $L_2(A)$ is equal to its area, \emph{i.e.} the two-dimensional Lebesgue measure. Taking inspiration from the unidimensional framework, the $L_2$  functional is also called sojourn time, although no time is involved in this context. \\

\textbf{Notations.}
All over the paper, $\|\cdot \|$ denotes the Euclidean norm in $\R^2$ and $I_2$ the $2 \times 2$ identity matrix. \smallskip

We will also denote by $|\cdot|$ the two-dimensional Lebesgue measure of any Borelian set in $\R^2$ and by $|\cdot|_1$ its one-dimensional Hausdorff measure. In particular, when $T$ is a bounded rectangle in $\R^2$ with non empty interior,
\begin{align*} 
 L_{0}(T)=1, \quad L_{1}(T)=\frac 12 |\partial T|_1, \quad  L_{2}(T)=|T|,
\end{align*}
where $\partial T$ stands for the boundary of $T$. \smallskip

Let $T$ be a bounded rectangle in $\R^2$ with non empty interior.  In the following notation $T\nearrow \R^2$ stands for the limit along any sequence of bounded rectangles that grows to $\R^2$.  For that, set $N>0$ and define $$T^{(N)}:=\big\{Nt: t\in T\big\}$$ the image of a fixed rectangle $T$ by the dilatation $t\mapsto Nt$; then letting $T \nearrow \R^2$ is equivalent to $N\to\infty$. Remark that $T^{(N)}$ is a Van Hove (VH)-growing sequence (see Definition 6 in \citet{bulinski2012central}),  \emph{i.e.}, $|\partial T^{(N)}|_{1}/|T^{(N)}|\to0$ as $N\to\infty$.  In the sequel, we sometimes drop the dependency in $N$ of the rectangle $T$ to soften notation.

\medskip

We now define the main notions that we will deal with.

\begin{Definition}[Considered Gaussian field] \label{defFGT}
Let  $g$   be a Gaussian random field  defined on $\R^2$ that is
\begin{itemize}
\item stationary, isotropic with $\E[g(0)]=0$,  $\Var g(0)=\sigma^2_g$, $\Var g'(0)= \lambda I_2$ for some $\lambda>0$, $\sigma_g > 0$,
\item whose covariance function $r(t) = Cov(g(0), g(t))$ is $\mathcal C^4$ and satisfies
$$| r(t) |  = O(\| t \|^{-\alpha}),~\mbox{ for some}~\alpha > 2~\mbox{ as }~\|t\|\to\infty.$$
\end{itemize}
\end{Definition}

We will consider perturbations of the above Gaussian field prescribed by the following.

\begin{Definition}[Perturbed  Gaussian  field]\label{MODELadditive}
Let  $X$  be a  random variable such that $\E[|X|^3] < + \infty$ and $\E[X]=0$. Let  $g$  be a Gaussian random field as in  Definition \ref{defFGT},  with  $X$ independent of $g$. We consider the following perturbed field
\begin{center}
$f(t) = g(t) + \epsilon \, X,~ t\in \R^2~$,   with $\epsilon >0$.
\end{center}
\end{Definition}

Let   $u \in \R$ and $T$ a bounded rectangle in $\R^2$. For $h$ any real-valued stationary  Gaussian  random field, we consider the excursion set within $T$ above level $u$:
\[\{t\in T\,:\,h(t)\ge u\}=T\cap E_{h}(u),\quad \mbox{where }E_{h}(u):= h^{-1}([u,+\infty)).\]

We now introduce the  Lipschitz-Killing curvatures for the excursion set $E_{h}(u)$, $u\in\R$ (see  \citet{AT07},  \citet{BD16}, \citet{BDDE} for more details).

\begin{Definition}[LK curvatures of $E_{h}(u)$] \label{defLK}
Let $h$ be a real-valued stationary Gaussian field that is almost surely of class $\mathcal C^2$. Define the following Lipschitz-Killing curvatures for the excursion set $T\cap E_{h}(u)$, $u\in\R$, $T$ bounded rectangle in $\R^2$,
\begin{align}
L_{2}(h,u,T) : &=  |T\cap E_{h}(u)| \nonumber\\
L_{1}(h,u,T) : &= \frac{|\partial (T\cap E_{h}(u))|_{1}}{2},   \nonumber  \\
L_{0}(h,u,T) : &=\sharp  \mbox{ connected components in } T\cap E_{h}(u)  -  \sharp \mbox{ holes in }  T\cap E_{h}(u)  \nonumber.
\end{align}
Furthermore,  the normalized LK curvatures are given  by
\begin{align*} 
C_i^{/T}(h,u):=\frac{L_i(h,u,T)}{|T|},~\mbox{ for }i=0,1,2,
\end{align*}
and  the  associated  LK  densities are
\begin{align*}  
C^{\ast}_i(h,u) := \underset{T\nearrow \R^2}{\lim}\,\E[C^{/T}_i(h,u)],~\mbox{ for } i=0,1,2.
\end{align*}
\end{Definition}

Figure \ref{campi} displays  a realization of a Gaussian random field  (first row) and of the associated perturbed one (second row) and two excursion sets for these fields for $u = 0$ (center) and $u = 1$ (right). We chose here a Student distributed centered  random variable $X$  with $\nu=5$ degrees of freedom, \emph{i.e.}, $X \sim t(\nu=5)$, and $g$ a Bergmann-Fock Gaussian field prescribed by its covariance function $r(t)= \sigma_g^2\, \rme^{-\kappa^2 \parallel t \parallel^2}$.\\

In Figure \ref{campi} one can appreciate a visual similarity between these  images and in particular in terms of their excursion sets. Then it could be difficult to evaluate the perturbation behind the considered  Gaussian model  by looking exclusively at Figure \ref{campi}. This motivate the necessity of an image processing  in order to measure the impact of the perturbation.  The goal of the next section will be to  study the LK curvatures of the perturbed field in order to both quantify the discrepancy between these black-and-white images and evaluate the robustness  with respect to a small perturbation  of the considered   geometrical characteristics of the excursion sets.\\

\begin{figure}[H] \centering
\includegraphics[width=4.4cm, height=4.5cm]{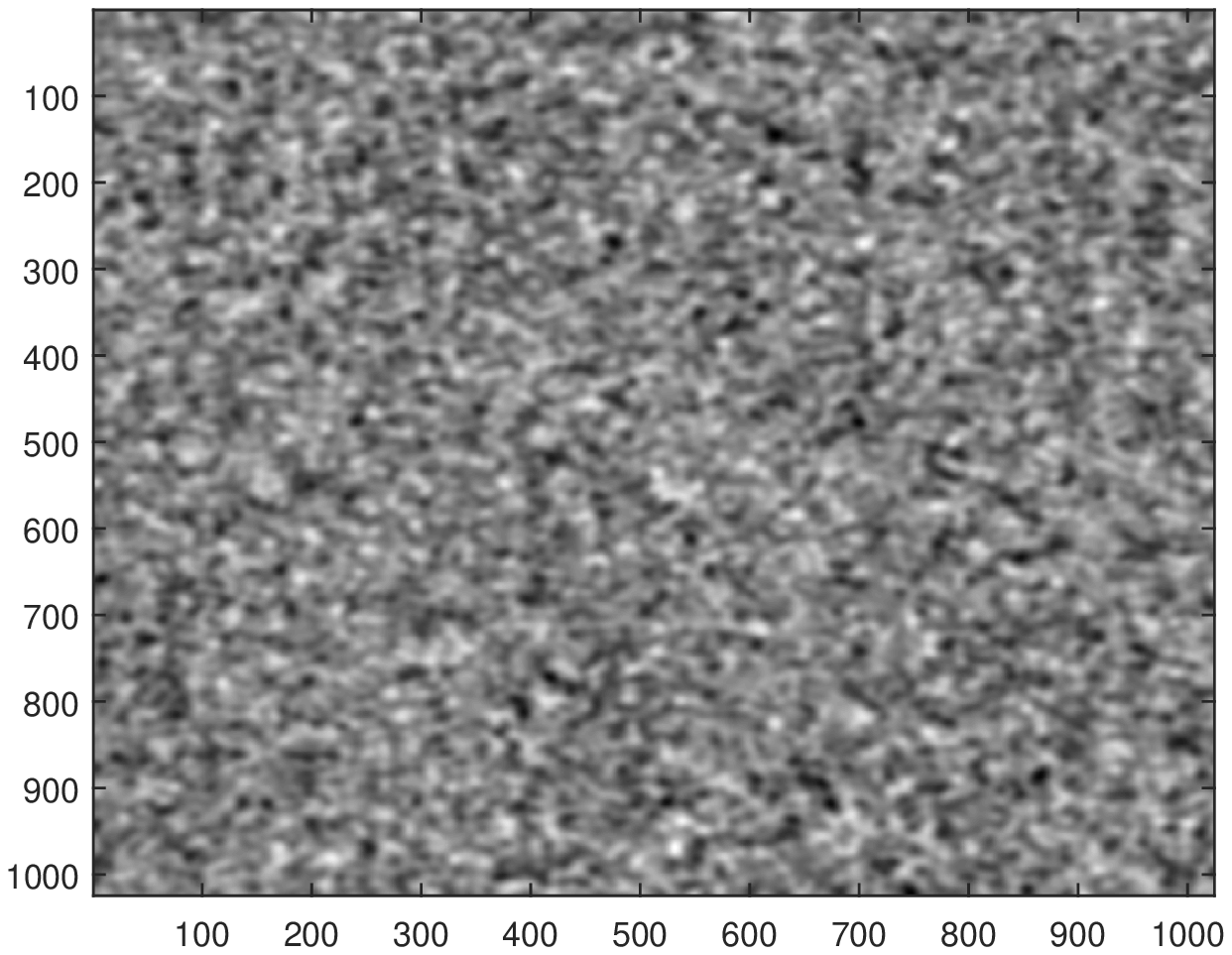}
\includegraphics[width=4.4cm, height=4.5cm]{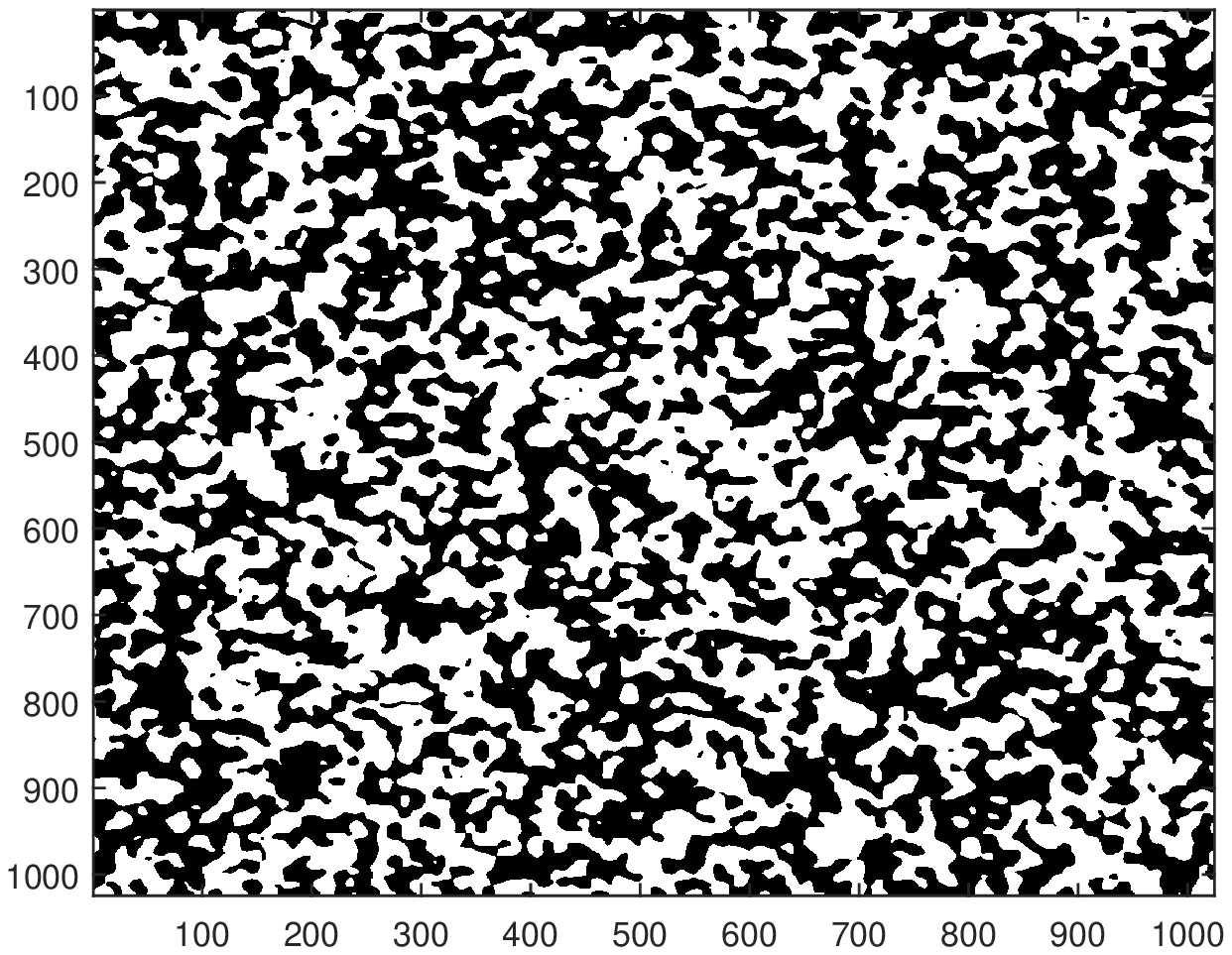}
\includegraphics[width=4.4cm, height=4.5cm]{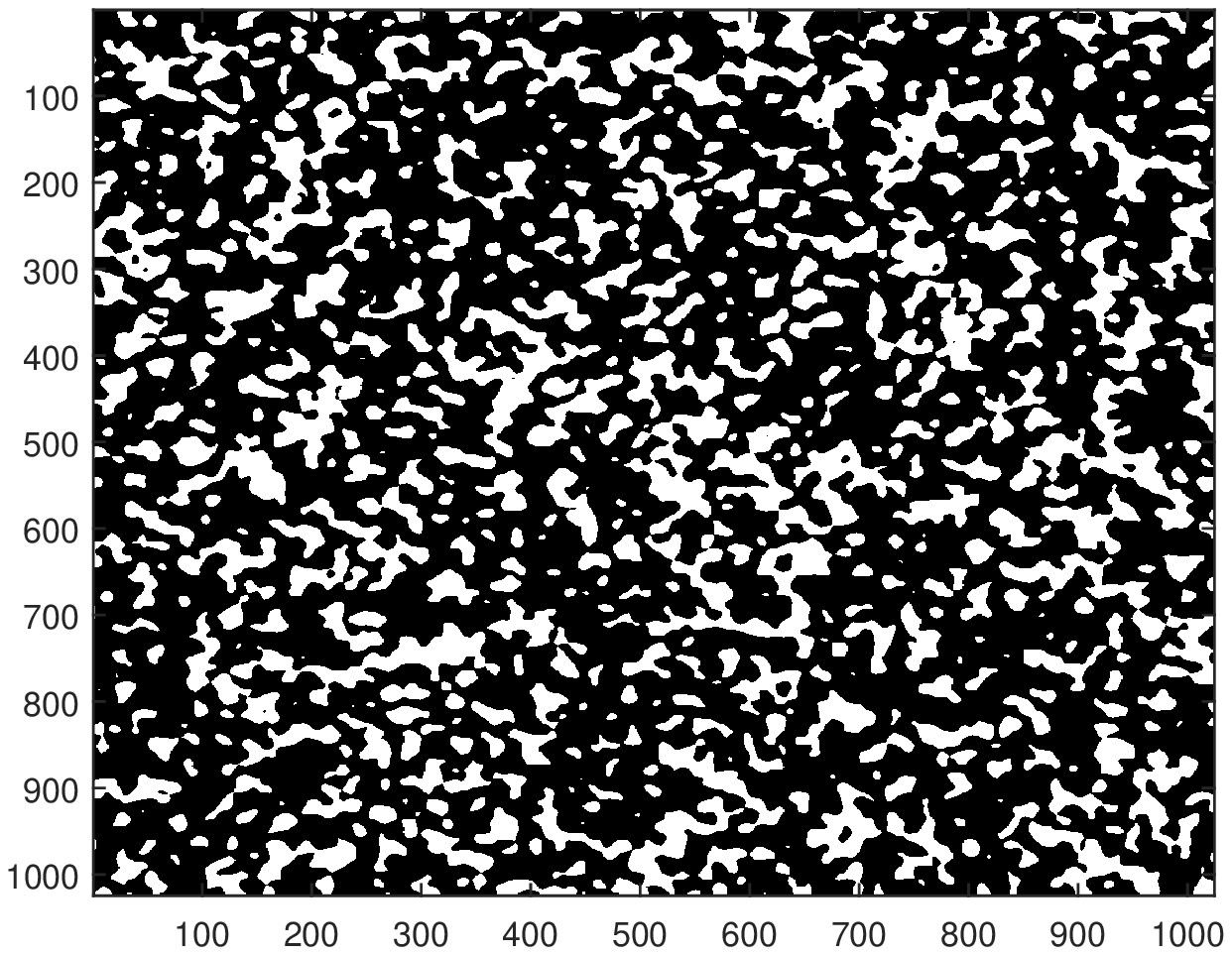}\\
\includegraphics[width=4.4cm, height=4.5cm]{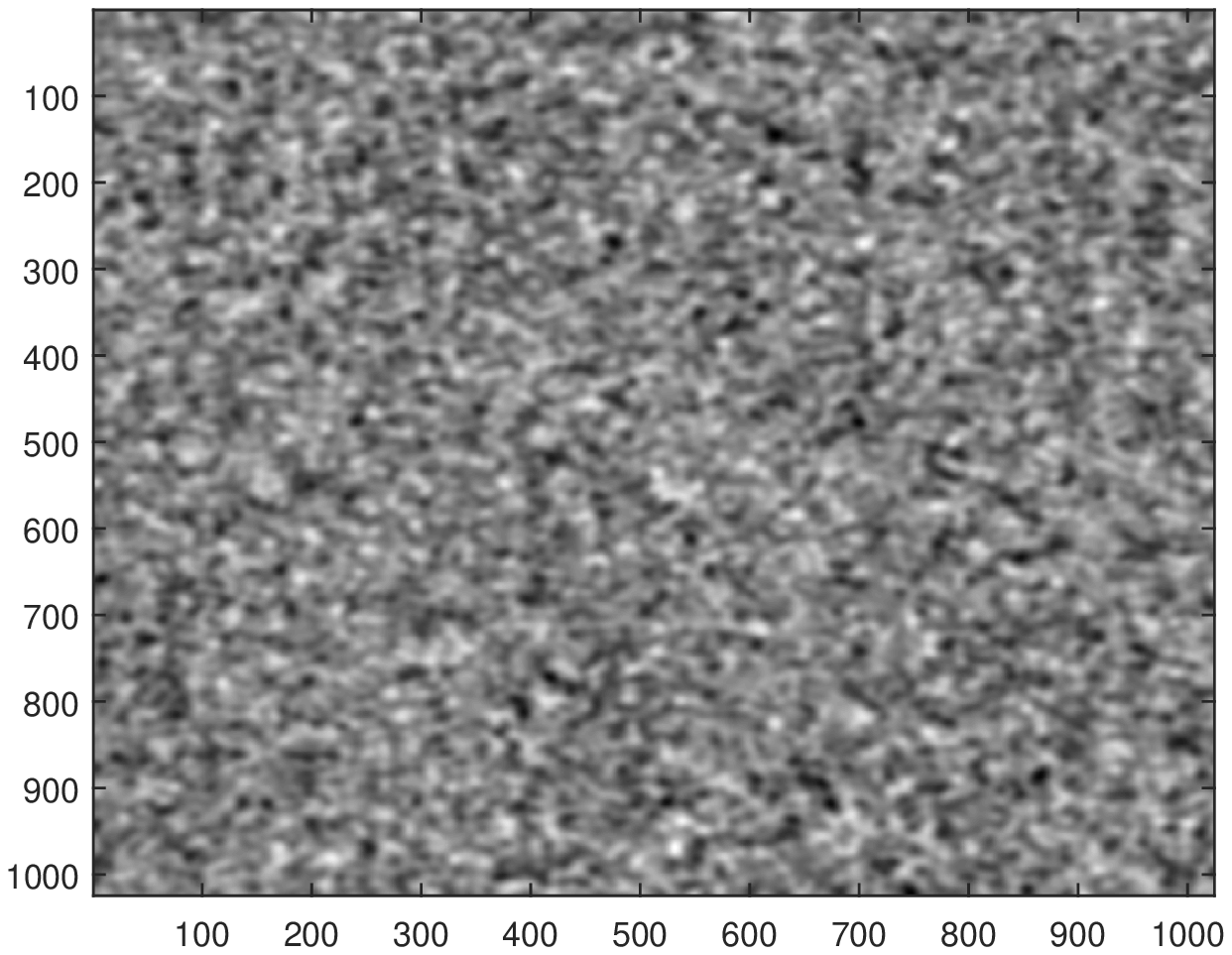}
\includegraphics[width=4.4cm, height=4.5cm]{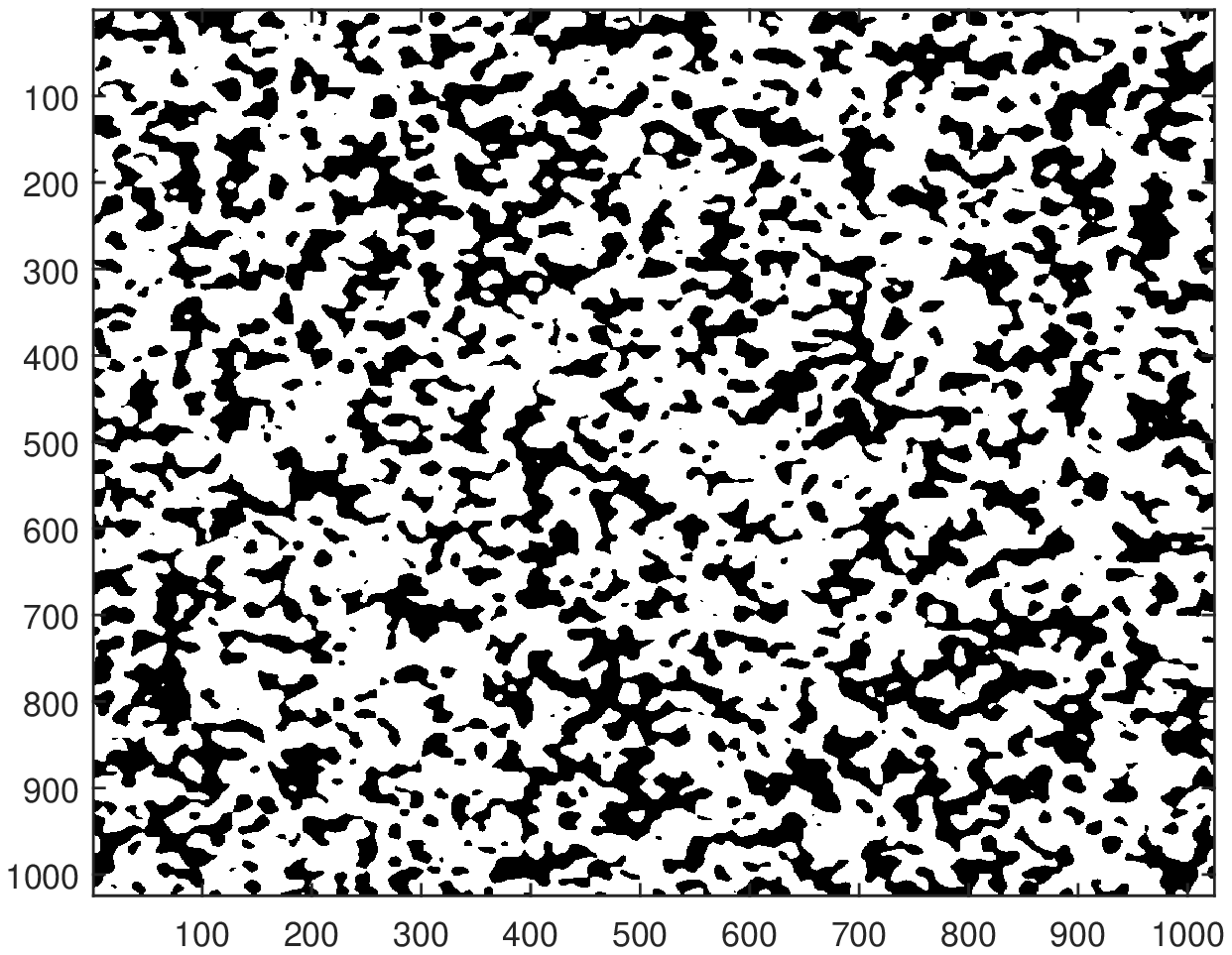}
\includegraphics[width=4.4cm, height=4.5cm]{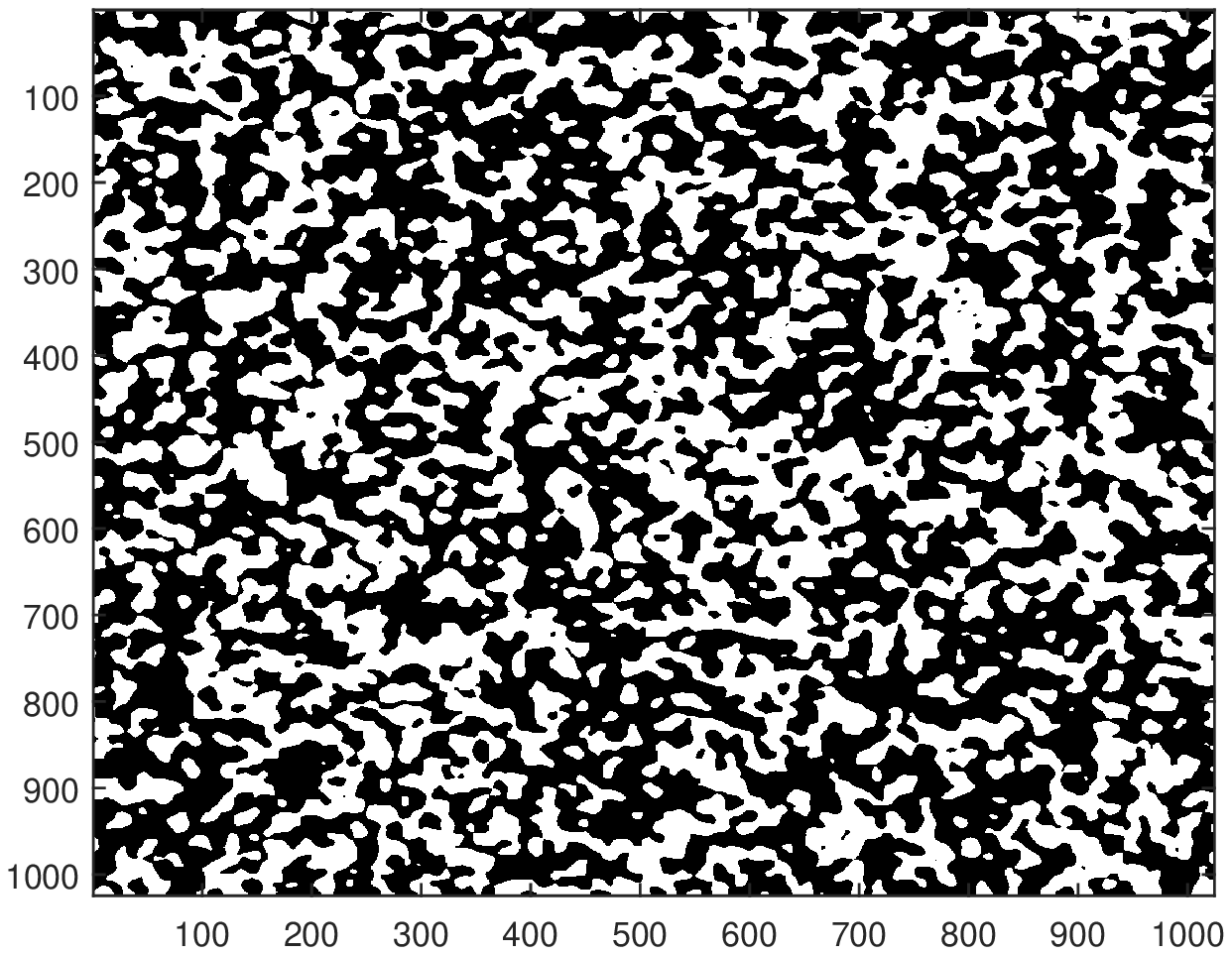}
  \vspace{-0.4cm}
\caption{\textbf{Gaussian random field and its perturbed counter-part} as in Definitions \ref{defFGT} and \ref{MODELadditive} with covariance $r(t) = \sigma_g^2 \, \rme^{-\kappa^2 \parallel t \parallel^2}$,   for $\sigma_g = 2$, $\kappa = 100/2^{10}$ in a domain of size $2^{10} \times 2^{10}$ pixels, with $\epsilon =1$ and $X \sim t(\nu=5)$. \textbf{First row}:  A   realization of Gaussian random field $g$ (left) and the  two  associated excursion sets for $u = 0$ (center) and $u = 1$ (right).   \textbf{Second row}: The   associated  realization of a perturbed Gaussian random field $f$ (left) and two excursion sets for $u = 0$ (center) and $u = 1$ (right).
 }\label{campi}
\end{figure}

\subsection{Mean LK curvatures of excursion sets of perturbed  Gaussian  model}\label{LKmean}

Let $g$ be as in Definition  \ref{defFGT}. The Gaussian kinematic formula provides the mean LK curvatures of excursion sets of $g$ within a rectangle $T$ (see, \textit{e.g.}, Theorem 13.2.1 in \citet{AT07} or Theorem 4.3.1 in \citet{AT11}), for $u\in\R$,
\begin{eqnarray}\label{eqGauss}
\mathbb E[C^{/T}_2(g,u)] &=& \Psi\left( \frac{u}{\sigma_g} \right);\cr
\mathbb E[C^{/T}_1(g,u)] &=& \Psi\left( \frac{u}{\sigma_g} \right) \frac{|\partial T|_1}{2|T|} - \frac{\sqrt{2\pi \lambda}}{4\, \sigma_g} \, \Psi'\left ( \frac{u}{\sigma_g}  \right );\cr
\mathbb E[C^{/T}_0(g,u)] &=& \Psi\left( \frac{u}{\sigma_g} \right) \frac{1}{|T|}  - \sqrt{\frac{\lambda}{2\pi\, \sigma_g}}\, \Psi'\left ( \frac{u}{\sigma_g}  \right )  \frac{|\partial T|_1}{2|T|} + \frac{\lambda}{2\pi\, \sigma_g^2} \,\Psi^{''} \left( \frac{u}{\sigma_g} \right),
\end{eqnarray}
being $\lambda$ the second spectral moment of $g$ and $\Psi(x): = \frac{1}{\sqrt{2\pi}} \int_x^{+\infty} \text{e}^{-t^2/2}\,dt$ the  Gaussian tail distribution with zero mean and unit variance.   Then the LK densities for the considered  Gaussian  field are given by
\begin{align}\label{eqGaussD}
C^*_2(g,u) = \Psi\left( \frac{u}{\sigma_g} \right), \quad C^*_1(g,u) = - \frac{\sqrt{2\pi \lambda}}{4\, \sigma_g}  \Psi'\left ( \frac{u}{\sigma_g}  \right ),  \quad
C^*_0(g,u) = \frac{\lambda}{2\pi\, \sigma_g^2} \Psi^{''} \left( \frac{u}{\sigma_g} \right).
\end{align}

\begin{Proposition}[LK curvatures for the perturbed  Gaussian  model]\label{LKadditive}
Let  $f(t) = g(t) + \epsilon \, X$,   $t\in \R^2$ as in Definition \ref{MODELadditive}.
Then, for small $\epsilon$, it holds that
 \begin{align}
\E[C^{/T}_0(f,u)] & =  C_0^*(g,u)\left( 1 +  \frac{\epsilon^2 \mathbb E[X^2]}{2\sigma^2_g} \left(H_2\left(\frac{u}{\sigma_g}\right) -2\right)\right) \label{C0Tadditive}\\
 & +  \frac{1}{\pi} C^*_1(g,u)\left(1+ \frac{\epsilon^2 \E[X^2]}{2 \sigma_g^2}H_2\left(\frac{u}{\sigma_g}\right)\right) \frac{|\partial T|_1}{|T|}    \nonumber \\
  & +    \left(C^*_2(g,u)+  \epsilon^2 \E[X^2] \frac{\pi}{\lambda} C^*_0(g,u)\right)\frac{1}{|T|}+  O\left (\epsilon^3 \left( 1 + \frac{|\partial T|_1}{2|T|}   + \frac{1}{|T|}  \right )\right ), \nonumber  \\
\E[C^{/T}_1(f,u)] & =     C^*_1(g,u) + C^*_2(g,u) \frac{|\partial T|_1}{2|T|} \label{C1Tadditive} \\
& + \epsilon^2 \E[X^2] \left( \frac{C_1^*(g,u)}{2\sigma_g^2} H_2\left(\frac{u}{\sigma_g}\right) + C_0^*(g,u)\frac{\pi }{\lambda}   \frac{|\partial T|_1}{2|T|}\right) 
+   O\left (\epsilon^3 \left( 1 + \frac{|\partial T|_1}{2|T|}    \right ) \right ), \nonumber     \\
\E[C^{/T}_2(f,u)] & = C^*_2(g,u) + \epsilon^2 \E[X^2]   \frac{\pi}{\lambda} C^*_0(g,u) + O(\epsilon^3),    \label{C2Tadditive}
\end{align}
where   $H_2(y)= y^2-1$, for $y \in \R$  (i.e., the second Hermite polynomial) and the constants involved in the $O$-notation only depend on $g$ and $X$.
\end{Proposition}

\begin{proof}[Proof of Proposition \ref{LKadditive}]
Let $a: =\frac{ \lambda}{\sigma_g^2}$. In the following we will use that $\Psi'(x) = - \frac{1}{\sqrt{2\pi}}  \text{e}^{-x^2/2}$, $\Psi^{''}(x) = x \frac{1}{\sqrt{2\pi}}  \text{e}^{-x^2/2}$, $\Psi^{'''}(x) = \Psi'(x) H_2(x)$ and $\Psi^{''''}(x) = \Psi^{''}(x)H_2(x) + 2x\Psi'(x)= \Psi^{''}(x)(H_2(x) - 2)$.
From (\ref{eqGauss}),
Taylor developing the Gaussian tail distribution $\Psi$ and bearing in mind that $X$ is a centered random variable we have
\begin{align}\label{area_pert}
 \E[C^{/T}_2(f,u)] &= \E [ \E[C^{/T}_2(g,u-\epsilon X)| X] ] = \E [\Psi((u-\epsilon X)/\sigma_g)]\nonumber  \\  
 &= \Psi(u/\sigma_g)+ \frac{\epsilon^2}{\sigma_g^2}
\frac{\Psi^{''}(u/\sigma_g)}{2} \E[X^2]  + O\left (\frac{\epsilon^3}{\sigma_g^3} \E[|X|^3]\right),
 \end{align}
 where the constant involved in the O notation is absolute. One can  rewrite \eqref{area_pert}, by using the kinematic formula  for LK densities $C^*_i(g, u)$  of the Gaussian field $g$ in \eqref{eqGaussD}. Hence the result in \eqref{C2Tadditive}.
Analogously,
\begin{align*}
 \E[C^{/T}_1(f,u)]  & =  \E [ \E[C^{/T}_1(g,u-\epsilon X)| X] ] -\frac{\sqrt{a}}{4} \sqrt{2\pi}  \E[\Psi^{'}((u-\epsilon X)/\sigma_g)]\\
 &+ \E[\Psi((u-\epsilon X)/\sigma_g)] \frac{|\partial T|_1}{2|T|}\\
 & = -\frac{\sqrt{a}}{4} \sqrt{2\pi}   \left (  \Psi^{'}(u/ \sigma_g)  + \frac{\epsilon^2}{\sigma_g^2}\frac{\Psi^{'''}(u/ \sigma_g)}{2} \E[X^2]  + O(\epsilon^3) \right ) \\
  & + \left ( C^*_2(g,u) + \epsilon^2 C^*_0(g,u) \frac{\pi  \E[X^2]}{\lambda}  + O(\epsilon^3) \right ) \frac{|\partial T|_1}{2|T|}.
\end{align*}
Then by using the Gaussian LK densities $C^*_i(g, u)$ in \eqref{eqGaussD}, we get  Equation \eqref{C1Tadditive}.
Finally,
\begin{align*}
 \E[C^{/T}_0(f,u)] &= \E [ \E[C^{/T}_0(g,u-\epsilon X)| X] ] \\
 & =   \frac{a}{2\pi} \E[\Psi^{''}((u-\epsilon X)/\sigma_g)] - \sqrt{\frac{a}{2\pi}} \E[\Psi^{'}((u-\epsilon X)/\sigma_g)] \frac{|\partial T|_1}{2|T|} \\
 &+ \E[\Psi(u-\epsilon X)/\sigma_g)] \frac{1}{|T|}\\
  & =  \frac{a}{2\pi}  \left ( \Psi^{''}(u/\sigma_g) + \frac{\Psi^{''''}(u/\sigma_g)}{2} \frac{\epsilon^2}{\sigma_g^2} \E[X^2] + O(\epsilon^3)  \right )\\
 & - \sqrt{\frac{a}{2\pi}}  \left (  \Psi^{'}(u/ \sigma_g)  + \frac{\epsilon^2}{\sigma_g^2}\frac{\Psi^{'''}(u/ \sigma_g)}{2} \E[X^2]  + O(\epsilon^3) \right ) \frac{|\partial T|_1}{2|T|}\\
  & +   \left ( \Psi(u/ \sigma_g) + \frac{\epsilon^2}{\sigma_g^2}
\frac{
\Psi^{''}(u/\sigma_g)}{2} \E[X^2]  + O(\epsilon^3) \right )\frac{1}{|T|}.
\end{align*}
As before, by using  \eqref{eqGaussD}, we get
\begin{align*}
 \E[C^{/T}_0(f,u)] = & \,  C_0^*(g,u)\left ( 1 + \epsilon^2(H_2\left(\frac{u}{\sigma_g}\right) -2) \frac{\mathbb E[X^2]}{2\sigma^2_g} \right )\cr
 &+ \frac{1}{\pi}  C_1^*(g,u)\left (  1 + \epsilon^2 \frac{\mathbb E[X^2]}{2\sigma^2_g}H_2\left(\frac{u}{\sigma_g}\right)         \right ) \frac{|\partial T|_1}{|T|}\cr
 &+ \left (    C_2^*(g,u) + \epsilon^2 \mathbb E[X^2]  \frac{\pi}{\lambda} C_0^*(g,u)         \right ) \frac{1}{|T|} + O\left (\epsilon^3 \left( 1 + \frac{|\partial T|_1}{2|T|}   + \frac{1}{|T|}  \right )\right ).
  \end{align*} 
\end{proof}

\begin{Remark}[Case of additive spatially variant perturbation]\rm
Notice  that the mean of LK curvatures in Proposition \ref{LKadditive} can be derived also in the case of  an additive spatially variant perturbation, \emph{i.e.},  $f(t) = g(t) + \epsilon \, X(t)$, for $ t\in \R^2$ and  $\epsilon >0$,  with $X$ a stationary random field with finite third moment and independent of $g$. The proof comes down in the same way and the results are completely analogous to those in Equations \eqref{C0Tadditive}, \eqref{C1Tadditive} and \eqref{C2Tadditive}. However, the asymptotics results obtained in Section \ref{CLTsection} will become more challenging in that case. Indeed,  even in the classical case of excursion area of Gaussian fields, to the best of our knowledge  we are not aware of any (quantitative) central limit theorem  in the case of a non-constant level.   This could represent an interesting point to investigate in a future work. For sake of completeness, the interested reader is referred to \citet{KratzLeon} where CLT results are obtained for the curve-crossings number of a stationary Gaussian process ($d=1$) according to the form of the moving curve (periodic or linear).
\end{Remark}

\begin{Corollary}[LK densities for the considered  perturbed  Gaussian  model]\label{LKstaradditive}
Under assumption of Proposition  \ref{LKadditive} and using the same notation, it holds that
\begin{eqnarray*}
C^*_0(f, u) & = &  C_0^*(g,u)\left( 1 +  \frac{\epsilon^2 \mathbb E[X^2]}{2\sigma^2_g} \left(H_2\left(\frac{u}{\sigma_g}\right) -2\right)\right ) + O(\epsilon^3),  \\
  C^*_1(f, u) & =  &  C^*_1(g,u)   \left(1 + \frac{\epsilon^2 \E[X^2]}{2\sigma_g^2}   H_2\left(\frac{u}{\sigma_g}\right)\right)   + O(\epsilon^3),\\
C^*_2(f, u) & =  &    C^*_2(g,u) + \epsilon^2 \E[X^2]   \frac{\pi}{\lambda} C^*_0(g,u)    + O(\epsilon^3).\\
\end{eqnarray*}
\end{Corollary}
The proof of Corollary \ref{LKstaradditive} is based on the property of the VH-growing sequence of rectangles $T$ on $\R^2$.

\begin{Remark}\label{remark:standard} \rm Let $u\in \R$. Notice that $f$  in Definition  \ref{MODELadditive}  is a  standard  random field in the sense of Definition 2.1 in \citet{BDDE}, then   it holds that
\begin{align*}
\E[C_{0}^{/T}(f,u)]&=C_{0}^{\ast}(f,u)+\frac{1}{\pi}C_{1}^{\ast}(f,u)\frac{|\partial T|_{1}}{|T|}+C_{2}^{\ast}(f,u)\frac{1}{|T|},\\
\E[C_{1}^{/T}(f,u)]&=C_{1}^{\ast}(f,u)+\frac12 C_{2}^{\ast}(f,u)\frac{|\partial T|_{1}}{|T|},\\
\E[C_{2}^{/T}(f,u)]&=C_{2}^{\ast}(f,u).
\end{align*}
As a product one can build the following  unbiased estimator of $C_i^{\ast}(f,u),\ i=0,1,2$
\begin{align}
\widehat{C}_{0,T}(f,u)&=C^{/T}_0(f,u)-\frac{|\partial T|_1}{\pi |T|}C^{/T}_1(f,u)+\left(\frac{1}{2\pi}\left(\frac{|\partial T|_1}{|T|}\right)^2-\frac{1}{|T|}\right)C^{/T}_2(f,u), \label{eq:UestC0}\\
 \widehat{C}_{1,T}(f,u)&=C^{/T}_1(f,u)-\frac{|\partial T|_1}{2 |T|}C^{/T}_2(f,u),\label{eq:UestC1}\\
\widehat{C}_{2,T}(f,u)&=C^{/T}_2(f,u) \label{eq:UestC2}.
 \end{align}
 \end{Remark}

An illustration for the finite sample performance of the proposed three unbiased  estimators $\widehat{C}_{0, T}(g, u)$,  $\widehat{C}_{1, T}(g, u)$  and  $\widehat{C}_{2,  T}(g, u)$ obtained by adapting  Equations  \eqref{eq:UestC0}-\eqref{eq:UestC2} to $g$,   is given in Figure \ref{figureLKGAUSSIAN}. In this case, $g$ has a covariance function  $r(t) = \sigma_g^2 e^{-\kappa^2 \parallel t \parallel^2}$. Analogously, a good statistical performance of  $\widehat{C}_{0, T}(f, u)$,  $\widehat{C}_{1, T}(f, u)$  and  $\widehat{C}_{2,  T}(f, u)$ in Equations  \eqref{eq:UestC0}-\eqref{eq:UestC2}  can be observed  in   Figure \ref{figureLKperturbed}.  \\

\begin{figure}[H] 
\includegraphics[width=4.4cm, height=4.5cm]{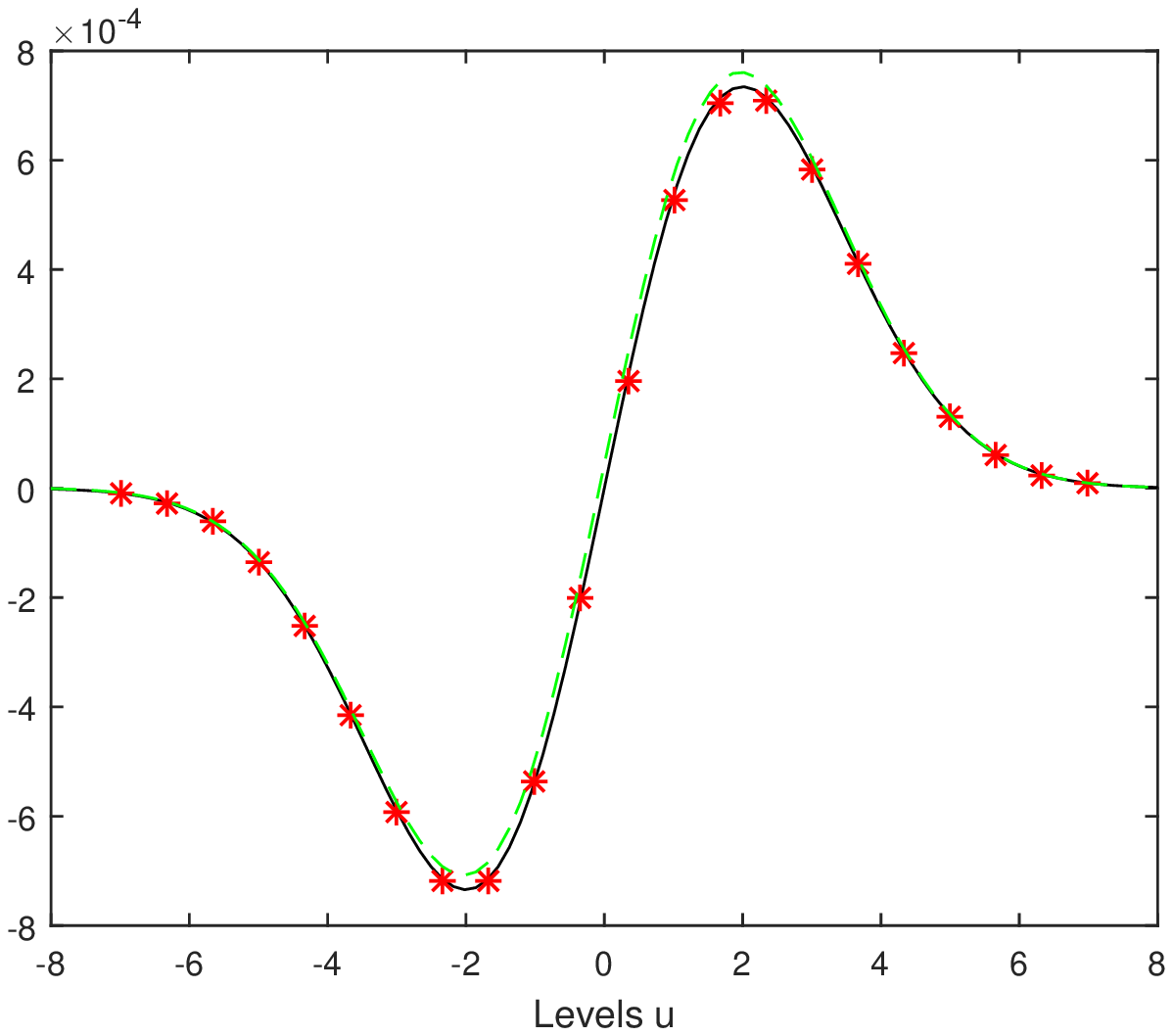}
\includegraphics[width=4.4cm, height=4.5cm]{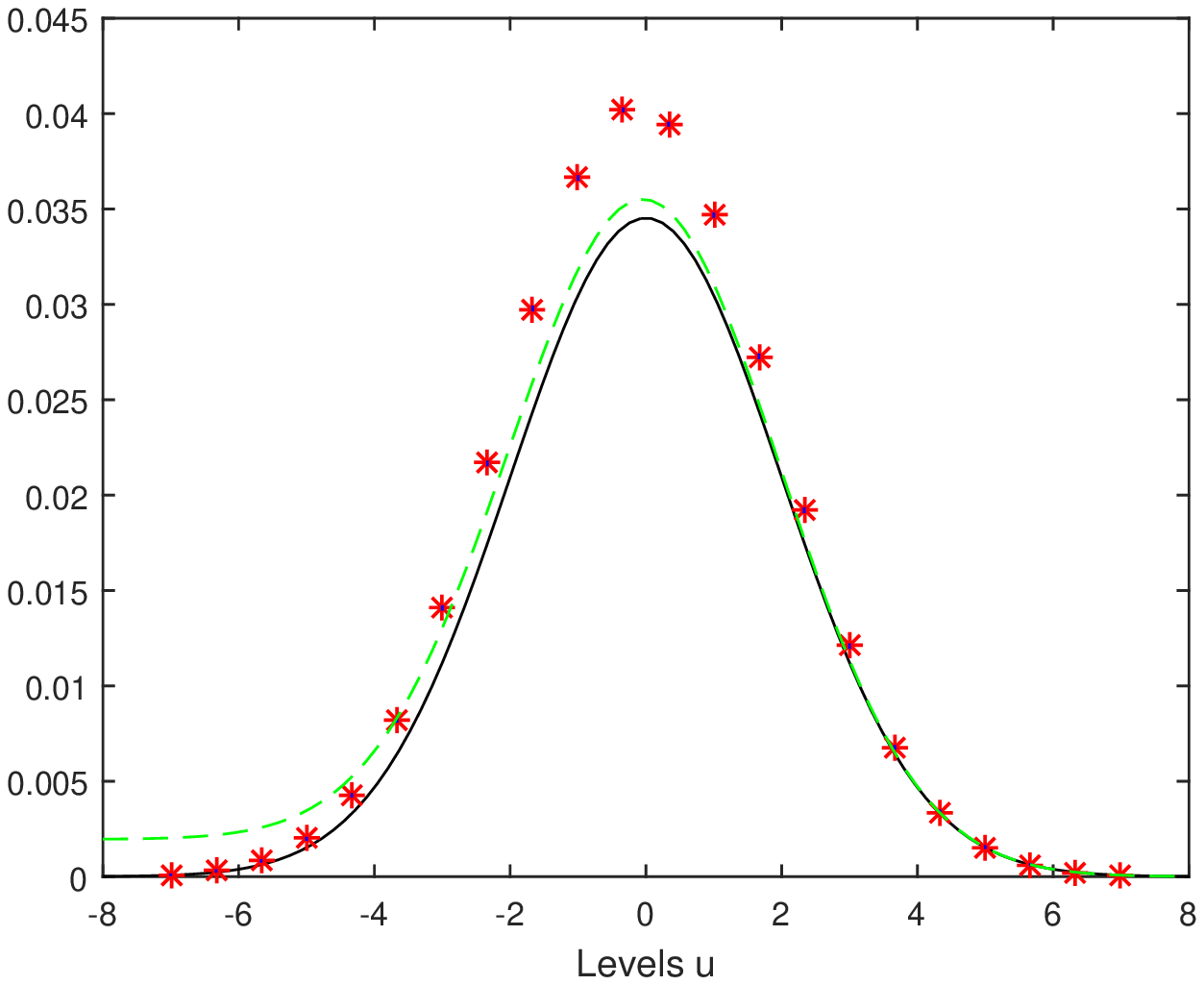}
\includegraphics[width=4.4cm, height=4.5cm]{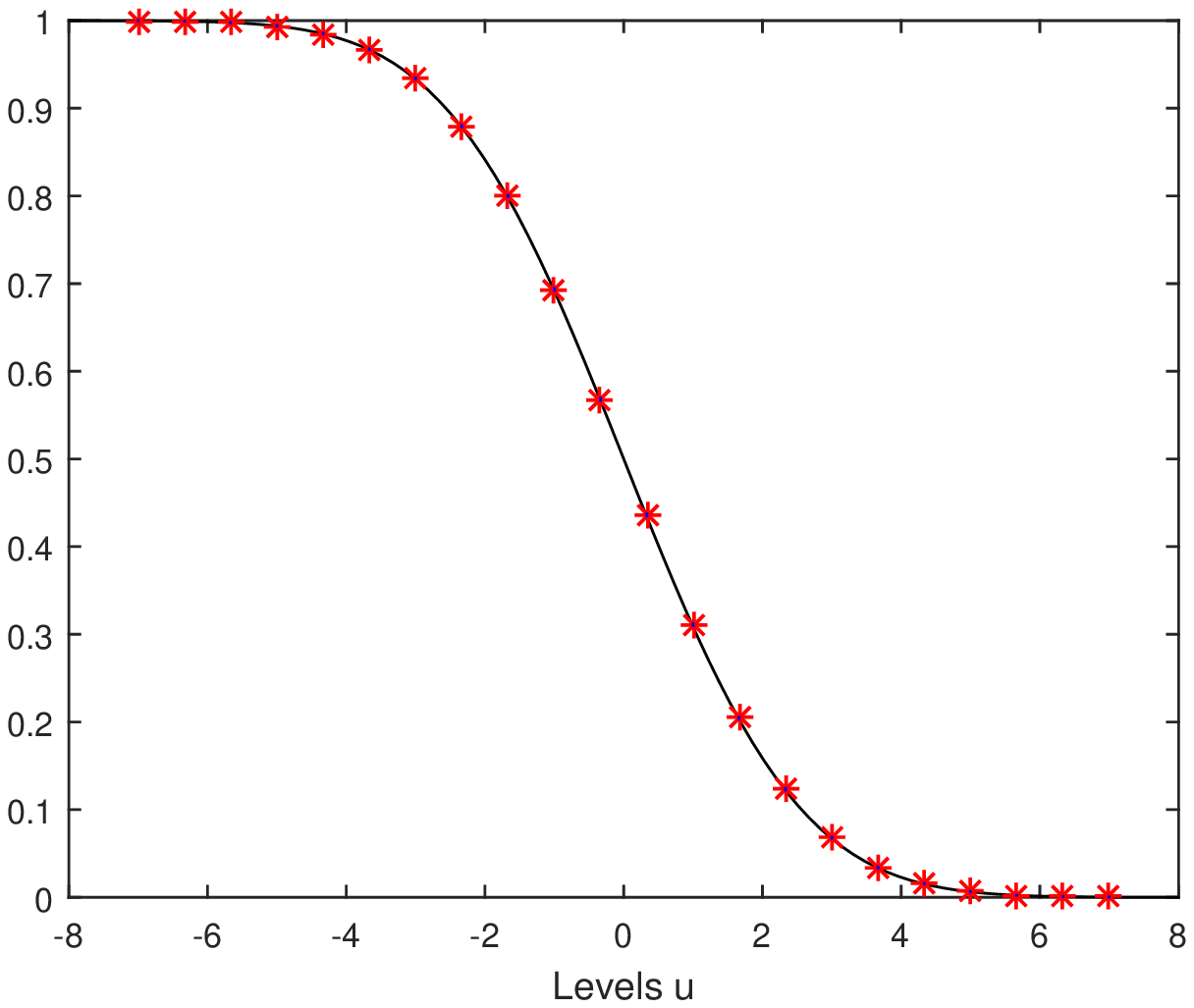}
  \vspace{-0.1cm}
\caption{\textbf{Gaussian random field} as in Definition \ref{defFGT} with covariance $r(t) = \sigma_g^2 e^{-\kappa^2 \parallel t \parallel^2}$, for $\sigma_g = 2$, $\kappa = 100/2^{10}$ in a domain of size $2^{10} \times 2^{10}$ pixels.  Theoretical $u\mapsto C^{\ast}_{0}(g, u)$ (left panel), $C^{\ast}_{1}(g, u)$ (center panel)  and $C^{\ast}_{2}(g, u)$ (right panel)  in \eqref{eqGaussD} are drawn in black lines.
We also display with red stars the averaged values on $M=100$  sample simulations of $\widehat{C}_{0, T}(g, u)$ (left panel),  $\widehat{C}_{1, T}(g, u)$ (center panel) $\widehat{C}_{2,  T}(g, u)$ (right panel),  obtained by adapting to $g$ estimators in  \eqref{eq:UestC0}-\eqref{eq:UestC2},  as a function of the level~$u$.  The  empirical intervals associated to the estimation of $\widehat{C}_{i,  T}(g, u)$, for $i=0,1,2$ are given by using red vertical lines.   These samples  have been obtained with \texttt{Matlab} using circulant embedding matrix.}\label{figureLKGAUSSIAN}
\end{figure}

The quantities $C^{/T}_0$, $C^{/T}_1$ and  $C^{/T}_2$ in   \eqref{eq:UestC0}-\eqref{eq:UestC2} are computed with the \texttt{Matlab} functions \texttt{bweuler},  \texttt{bwperim} and \texttt{bwarea}, respectively.   When it is required to specify the connectivity, we average between the 4th and the 8th connectivity. Since $C^{\ast}_{1}$ is defined as the average half perimeter, we divide by 2 the output derived from \texttt{bwperim}. From a numerical point of view, \texttt{bweuler} and \texttt{bwarea} functions seem very precise contrary to the  \texttt{bwperim} function which  performs less well (see center panels in Figures \ref{figureLKGAUSSIAN} and \ref{figureLKperturbed}). It was expected due to the pixelisation effect.  \\

Figures \ref{figureLKGAUSSIAN} and \ref{figureLKperturbed} (center) illustrates that $C^{/T}_1$ (green dashed line) does not well approximate $C^{\ast}_{1}$ (black plain line), especially for small levels $u$ and that the correction induced by $\widehat{C}_{1, T}$ (red stars)  in Remark \ref{remark:standard} improves the approximation. In Figures \ref{figureLKGAUSSIAN} and \ref{figureLKperturbed} (left), we provide an analogous bias correction   for the Euler characteristic  by using $\widehat{C}_{0, T}$ in Remark \ref{remark:standard}. However  in this case, the discrepancy  is less evident than in the perimeter case.   Finally, in  Figure \ref{figureLKperturbed}, we also display the functions $u \mapsto C^{\ast}_{i}(g, u)$, for $i=0,1,2$, by using blue dashed lines. These functions could be used as reference values to visually appreciate the discrepancy between the considered geometrical characteristics of the excursion sets of the Gaussian model (blue dashed lines) and the perturbed one (black plain lines for $C^{\ast}_{i}(f, u)$,  red stars for $\widehat{C}_{i, T}(f, u)$, for $i=0,1,2$).\\

\begin{figure}[H] 
\includegraphics[width=4.4cm, height=4.5cm]{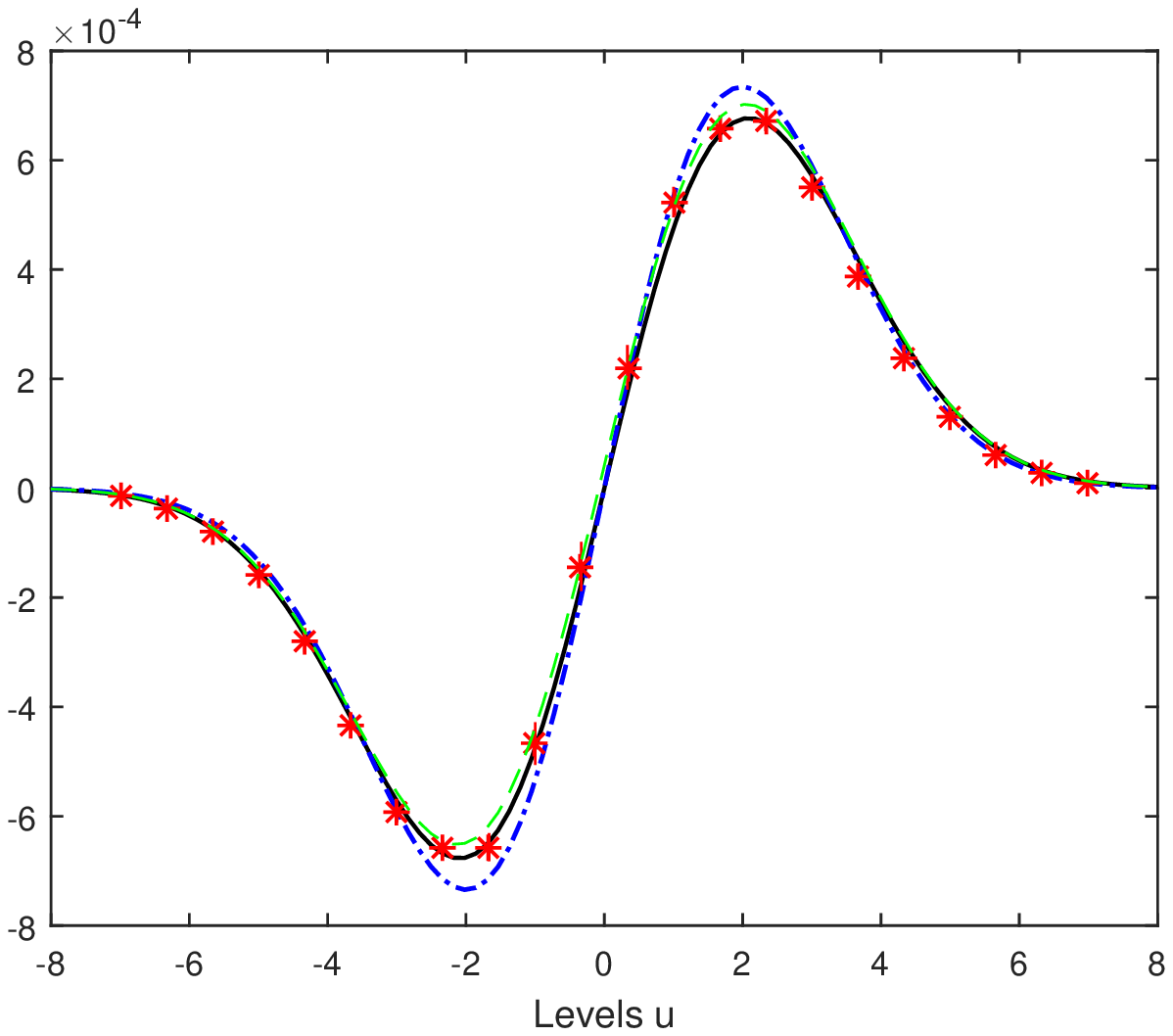}
\includegraphics[width=4.4cm, height=4.5cm]{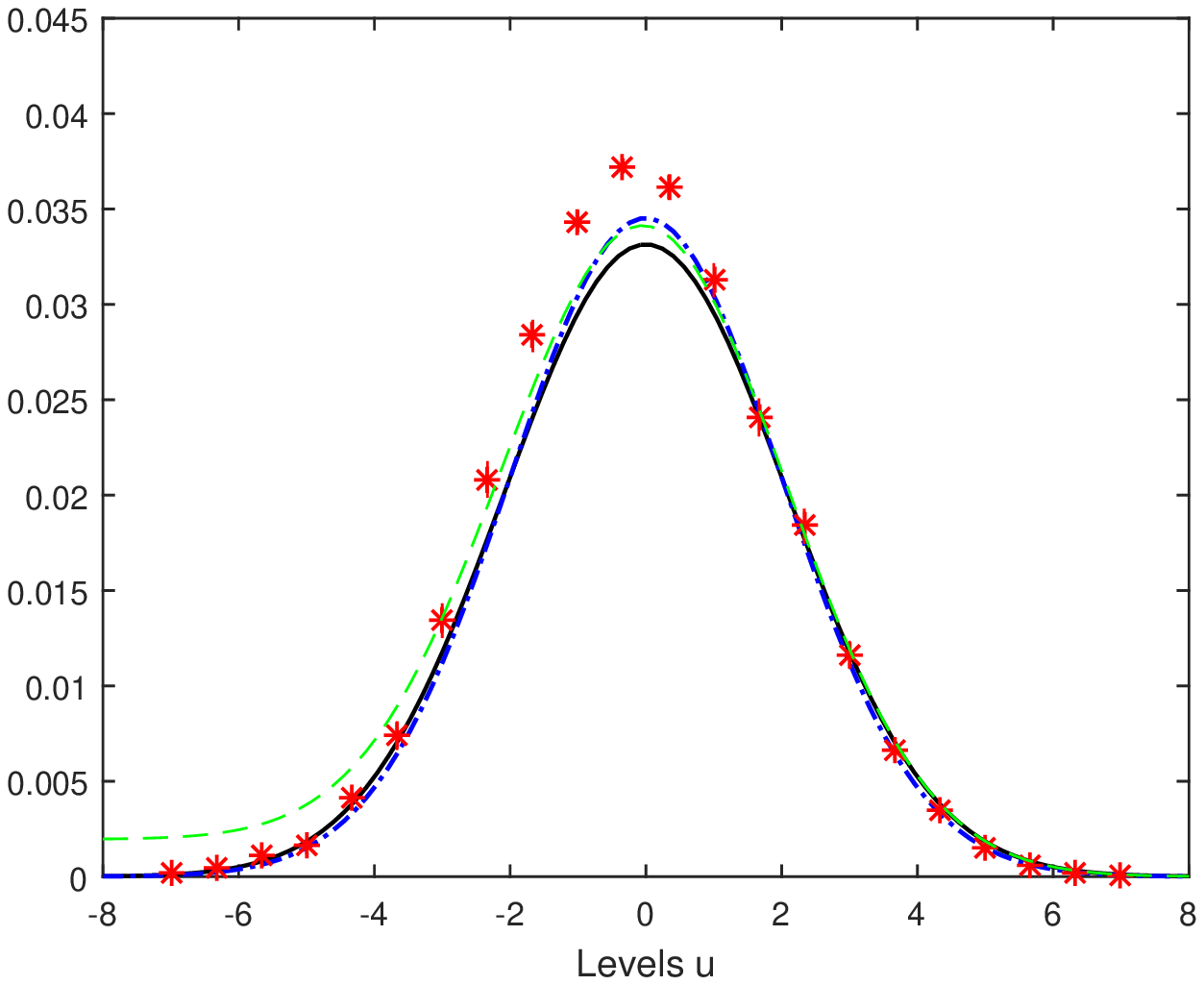}
\includegraphics[width=4.4cm, height=4.5cm]{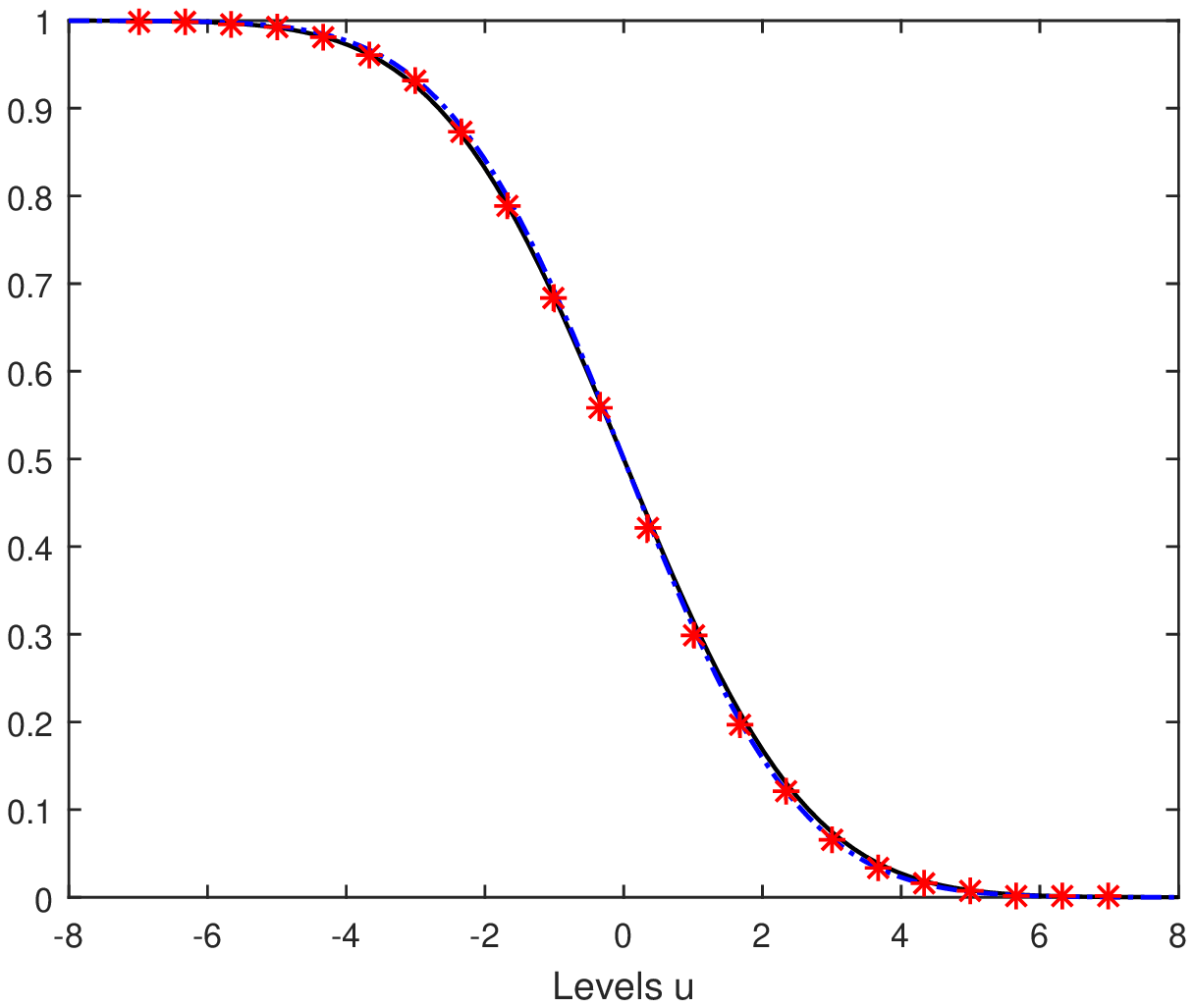}\\
\includegraphics[width=4.4cm, height=4.5cm]{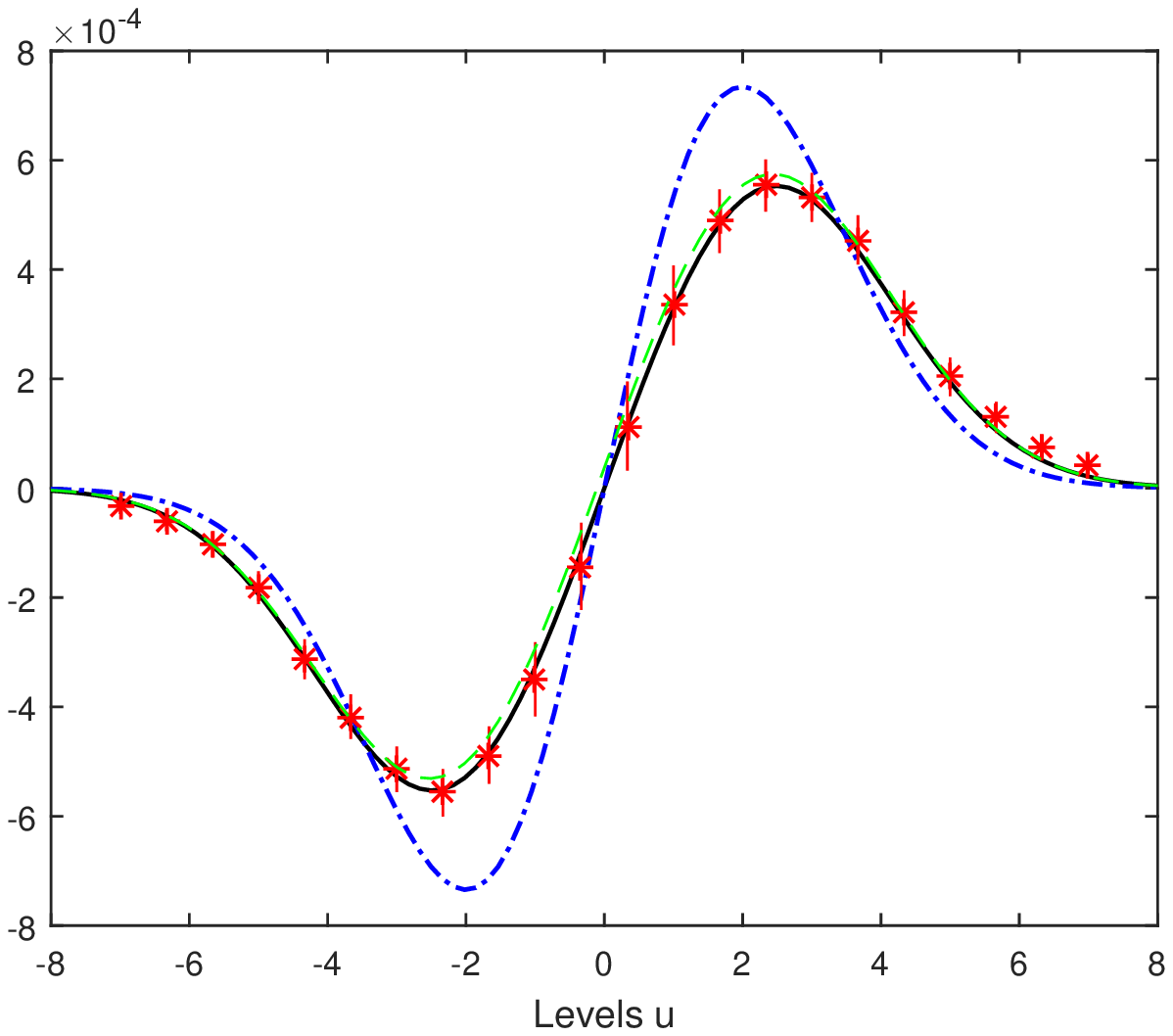}
\includegraphics[width=4.4cm, height=4.5cm]{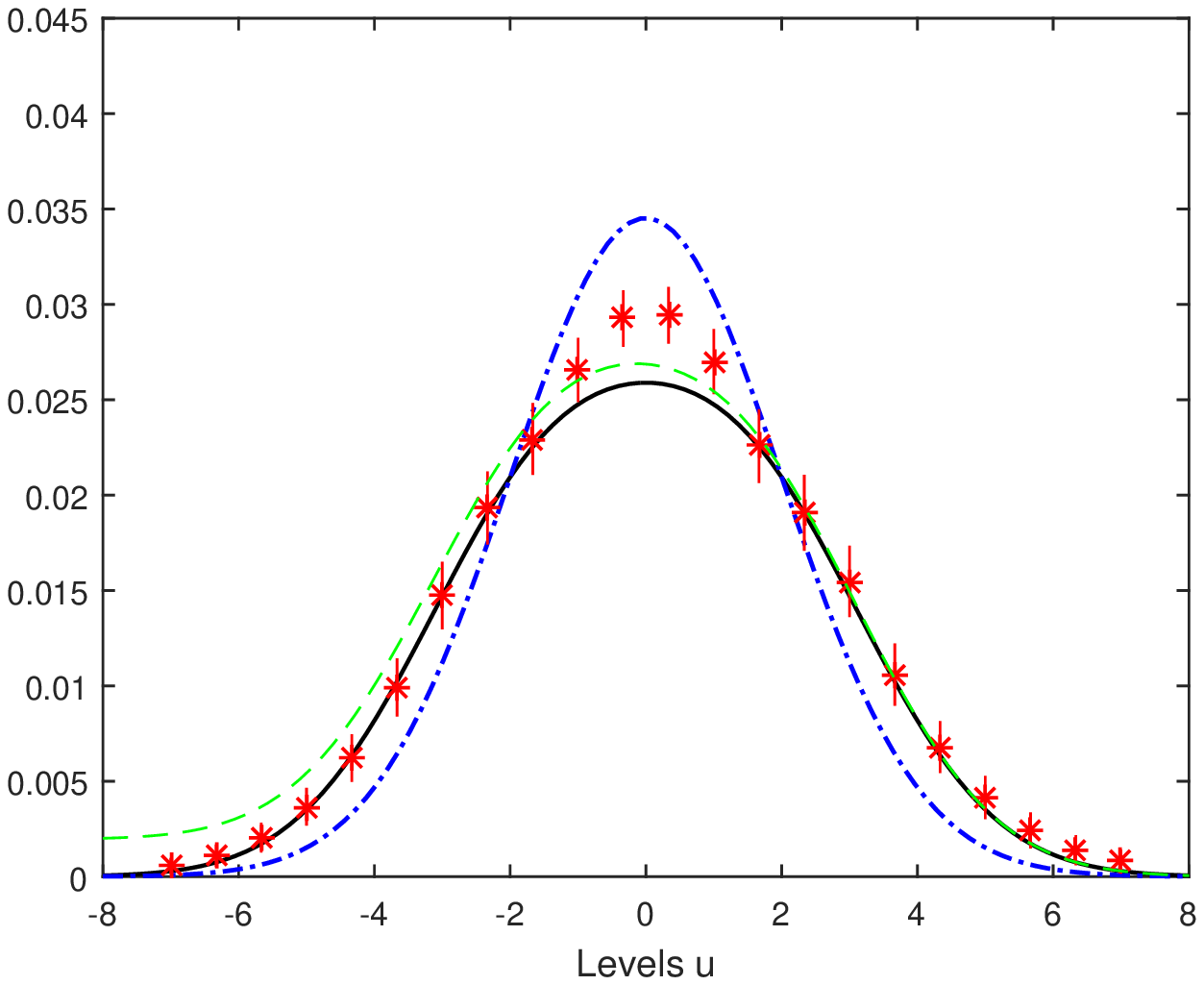}
\includegraphics[width=4.4cm, height=4.5cm]{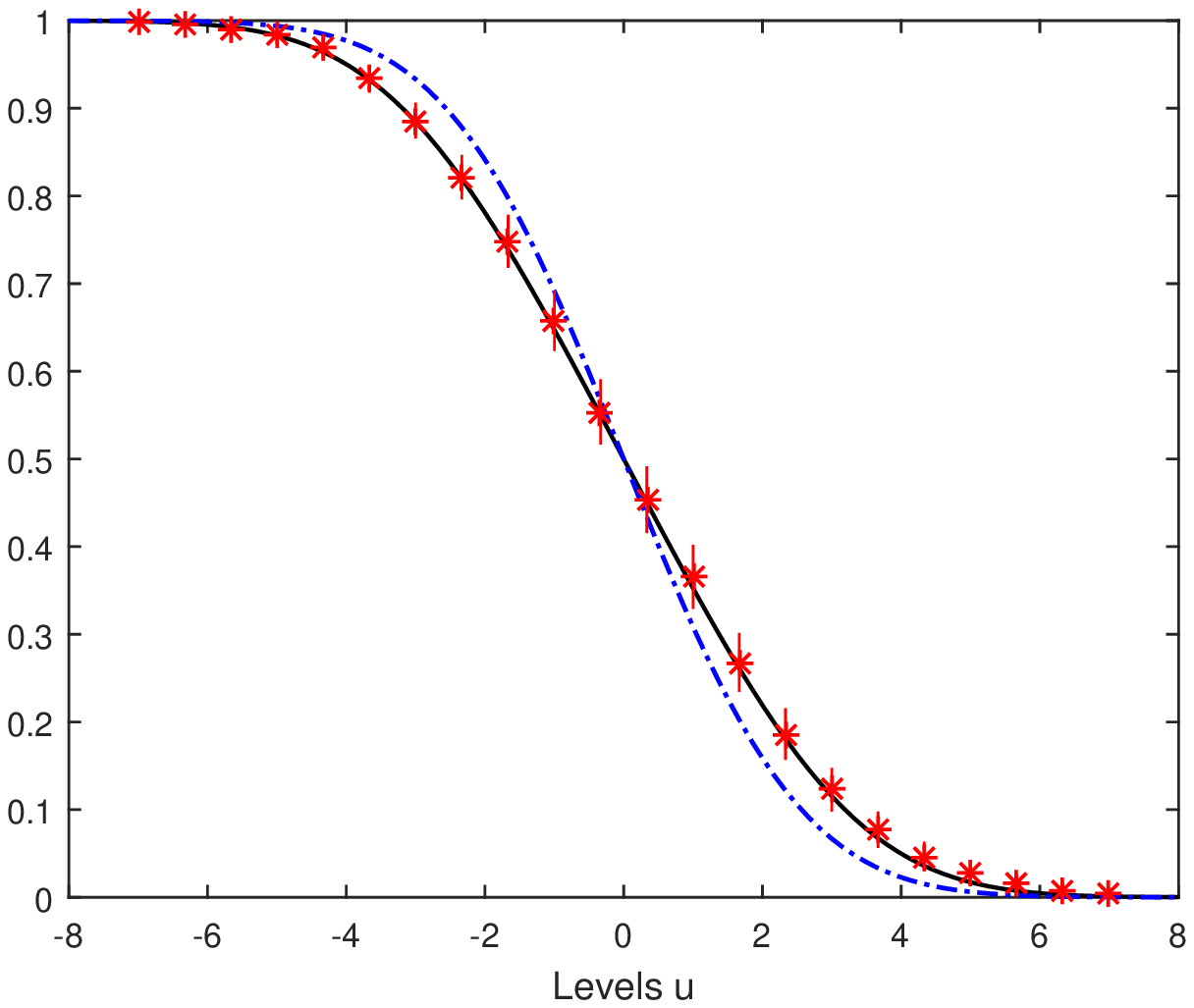}
  \vspace{-0.1cm}
\caption{\textbf{Perturbed Gaussian random field} as in Definition \ref{MODELadditive} with covariance $r(t) = \sigma_g^2 e^{-\kappa^2 \parallel t \parallel^2}$, for $\sigma_g = 2$, $\kappa = 100/2^{10}$ in a domain of size $2^{10} \times 2^{10}$ pixels, with  \textbf{$\epsilon =0.4$ and $X \sim t(\nu=5)$}  (\textbf{first row});  \textbf{$\epsilon = 1$ and $X \sim t(\nu=5)$} (\textbf{second row}).      Theoretical $u\mapsto C^{\ast}_{0}(f, u)$ (left panel), $C^{\ast}_{1}(f, u)$ (centered panel)  and $C^{\ast}_{2}(f, u)$ (right panel)  in Corollary \ref{LKstaradditive}  are drawn in black  plain  lines  and relative $u\mapsto C^{\ast}_{0}(g, u)$, $C^{\ast}_{1}(g, u)$  and $C^{\ast}_{2}(g, u)$ in blue dashed lines. We also  present   $u\mapsto C^{/T}_0(f, u)$ and $C^{/T}_1(f, u)$ in green dotted lines (left and center panels).
We also display the averaged values on $M=100$  sample simulations of $\widehat{C}_{0, T}(f, u)$ (left panel),  $\widehat{C}_{1, T}(f, u)$ (centered panel) $\widehat{C}_{2,  T}(f, u)$ (right panel)  in  \eqref{eq:UestC0}-\eqref{eq:UestC2},   as a function of the level~$u$  by using red stars.  The   empirical intervals associated to the estimation of $\widehat{C}_{i,  T}(f, u)$, for $i=0,1,2$ are given by using red vertical lines.   These samples  have been obtained with \texttt{Matlab} using circulant embedding matrix.}\label{figureLKperturbed}
\end{figure}

Conversely  to Figure \ref{campi} where the quantification of the perturbation was hard to get, by providing this image processing based on the LK curvatures,  we are now able to precisely measure the  impact of the perturbation.
Furthermore,  the contiguity of the Gaussian model $g$ with respect to the perturbed one $f$ can be observed in their LK curvatures when the magnitude of the perturbation  decreases, \emph{i.e.},  $\epsilon \rightarrow 0$ (see   Figure \ref{figureLKperturbed}: first row  with   $\epsilon = 0.4$,   second row with $\epsilon =1$).

\section{Asymptotics for the excursion area of perturbed Gaussian fields}\label{CLTsection}

Recall that $C^{/T}_2(f,u)=C^{/T}_2(g,u-\epsilon X)$ and that $\E[C^{/T}_2(g,u)]=C^*_2(g,u)=\Psi(u/\sigma_g)$.
We are interested in the asymptotic distribution as $T \nearrow \R^2$ of
\begin{equation} \label{YTu}
Y^{\epsilon}_T(u):=|T|^{1/2}\,\left(C^{/T}_2(f,u)-\E[C^{/T}_2(f,u)]\right).
\end{equation}

Considering the unperturbed case, by Theorem 3 in \citet{bulinski2012central} for instance, we know that for a Gaussian field $g$ as in Definition \ref{defFGT}  and for any $u\in \R$, the following convergence in distribution holds,
\begin{align} \label{TCLarea}
|T|^{1/2}\,\left(C^{/T}_2(g,u)-C^*_2(g,u)\right) \overset{d}{\underset{T\nearrow\R^2}{\longrightarrow}} \mathcal N(0, v(u)),
\end{align}
with
\begin{align}\label{vu}
v(u) = \frac{1}{2 \pi}\int_{\R^2} \int_0^{\rho(t)} \frac{1}{\sqrt{1-r^2}} \exp\left\{- \frac{u^2}{\sigma^2_g(1+r)}\right\} \rmd r\,\rmd t\end{align}
and $\rho(t) := corr(g(0), g(t))=r(t)/\sigma_g^2$. The interested reader is also referred to  \citet{KV16} and  \citet{Mu16}.\\

Actually, we are able to state a more powerful result. It is given in the next lemma where the convergence in \eqref{TCLarea} is proved to be uniform with respect to level $u$. In order to formulate our result, let us introduce the usual  Wasserstein distance between  random variables $Z_1$ and $Z_2$:
\begin{align*} 
d_W(Z_1, Z_2) =  \sup_{h\in   \mathcal H} \mathbb |\E[h(Z_1)] -\E[h(Z_2)]|,
\end{align*}
where $\mathcal H$ denotes the set of Lipschitz functions whose Lipschitz constant is $\le 1$.

\begin{Lemma} \label{lemunif}
Let $g$ be a Gaussian field as in Definition \ref{defFGT}. Then,
\[d_W \left (|T|^{1/2}\,\left(C^{/T}_2(g,u)-C^{*}_2(g,u)\right), \mathcal N(0,v(u))\right) = O\left ((\log |T|)^{-1/12} \right),\]
where $v(u)$ is defined in \eqref{vu} and the $O$-constant does not depend on the level $u$.
\end{Lemma}

\begin{proof}[Proof of Lemma \ref{lemunif}] We apply Lemma \ref{lem_app} that is postponed in the Appendix section since it is of some interest for its own. Indeed, the covariance function of the Gaussian field $g$ as in Definition \ref{defFGT} satisfies assumption in  \eqref{condPham}. Hence, with the notations that are in force, conclusion of Lemma \ref{lem_app} can be rewritten as
\begin{equation}\label{ineqDER}
d_W \left (|T|^{1/2}\,\left(C^{/T}_2(g,u)-C^{*}_2(g,u)\right), \mathcal N(0,\sigma^2(u/\sigma_g))\right) = O\left ((\log |T|)^{-1/12} \right),
\end{equation}
Note that obviously, \eqref{ineqDER} yields the same CLT as \eqref{TCLarea} for the excursion area as $T\nearrow \R^2$, and hence $\sigma^2(u/\sigma_g)$ equals variance $v(u)$ given by \eqref{vu}.
\end{proof}

Let us come back to the study of the asymptotics of $Y^{\epsilon}_T$, defined by \eqref{YTu}. We will use the next decomposition
\begin{align}
Y^{\epsilon}_T(u)&=|T|^{1/2}\left(C^{/T}_2(f,u)-C^*_2(g,u-\epsilon X) \right)
+|T|^{1/2}\left(C^*_2(g,u-\epsilon X)-\E[C^*_2(g,u-\epsilon X)] \right) \nonumber\\
&=: Z^{\epsilon}_T(u) + R^{\epsilon}_T(u).\label{Y=Z+R}
\end{align}

\subsection{Asymptotics for fixed small $\epsilon$ and $T\nearrow \R^2$}\label{asymtoticSmallEpsilon}

In this section, we introduce a non Gaussian random variable that we denote by $\Theta_\epsilon(u)$. We firstly provide an upper-bound for the Wasserstein distance between $Z^{\epsilon}_{T}(u)$  in \eqref{Y=Z+R}  and  $\Theta_\epsilon(u)$. Secondly, we describe the form of the  density of $\Theta_\epsilon(u)$ by providing a Taylor expansion for small $\epsilon>0$.

\begin{Theorem}[Quantitative asymptotics for $Z^{\epsilon}_{T}(u)$]\label{quantitative}
Let  $f(t) = g(t) + \epsilon \, X$,   $t\in \R^2$ as in Definition \ref{MODELadditive}.\\
For any fixed $\epsilon>0$ and $u\in \R$, we consider $\Theta_\epsilon(u)$ a random variable whose conditional distribution given $\{X=x\}$ is centered Gaussian with variance $v(u-\epsilon x)$, $v(\cdot)$ being defined by \eqref{vu}.
Then, as $T\nearrow \R^2$, it holds that
\begin{equation*} 
d_W(Z^{\epsilon}_{T}(u), \Theta_\epsilon(u)) = O \left((\log |T|)^{-1/12} \right ),
\end{equation*}
where the constant involved in the $O$-notation depends neither on $\epsilon$ nor on $u$.
\end{Theorem}

\begin{proof}[Proof of Theorem \ref{quantitative}]
By the definition of Wasserstein distance, we have
\begin{equation*}
\begin{split}
d_W(Z^{\epsilon}_{T}(u),\Theta_\epsilon(u))
&= \sup_{h\in \mathcal H} \mathbb E[|h(Z^\epsilon_T(u)) - h(\Theta_\epsilon(u))|] \cr
& \le  \mathbb E \left [ \sup_{h\in \mathcal H}\mathbb E[|h(Z^\epsilon_T(u)) - h(\Theta_\epsilon(u))| |X] \right ].
\end{split}
\end{equation*}
The latter supremum is equal to the Wasserstein distance between $Z^{\epsilon}_{T}(u)$ and $\Theta_\epsilon(u)$ with respect to the conditional expectation given $X$.\\
Actually, conditionally to $\{X=x\}$,
$Z^{\epsilon}_{T}(u)$ equals $|T|^{1/2}\,\left(C^{/T}_2(g,u-\epsilon x)-C^{*}_2(g,u-\epsilon x)\right)$ and
$\Theta_\epsilon(u)$ is $N(0,v(u-\epsilon x))$ distributed. \newline 
Hence, applying Lemma \ref{lemunif}  yields $
\sup_{h\in \mathcal H}\mathbb E[|h(Z^\epsilon_T(u)) - h(\Theta_\epsilon(u))| |X]= O \left((\log |T|)^{-1/12} \right)$,
where the $O$-constant does not depend on $u$ nor on $\epsilon$ and $X$. Lebesgue dominated convergence theorem allows us to conclude.
\end{proof}

\medskip

We now focus on the random variable $\Theta_\epsilon(u)$ that has been introduced in Theorem \ref{quantitative}. Let us quote that it is non Gaussian, yielding an unusual non Gaussian limit of $Z^{\epsilon}_{T}(u)$ as $T\nearrow \R^2$. In the next theorem, we provide the density distribution function of $\Theta_\epsilon(u)$ and a corresponding Taylor expansion for small $\epsilon > 0$.

\begin{Theorem} \label{theoTCLZ}
Under the same assumptions as Theorem \ref{quantitative}, it holds that, for fixed $\epsilon>0$,
\begin{align}\label{TCLZ}
Z^{\epsilon}_T=|T|^{1/2}\left(C^{/T}_2(f,u)-C^{*}_2(g,u-\epsilon X)\right)  \overset{d}{\underset{T\nearrow\R^2}{\longrightarrow}} \Theta_\epsilon(u),
\end{align}
where $\Theta_\epsilon(u)$'s probability density function is given by
\begin{equation} \label{eqhepsilon}
h_\epsilon:y\mapsto \E[\phi(v(u-\epsilon X),y)], \quad y\in \R,
\end{equation}
where $\phi(v,\cdot)$ stands for the p.d.f. of $N(0,v)$ and $v(\cdot)$ is given by \eqref{vu}. Furthermore,  $h_\epsilon$  can be expanded for small $\epsilon>0$ as
\begin{align}\label{TaylorHepsilon}
h_\epsilon(y) =  f_{\tiny \mbox{BEP}}^{\delta=0}(y) (1+ \gamma_1 - \gamma_2)+ f_{\tiny \mbox{BEP}}^{\delta=2}(y)(\gamma_2- 2 \gamma_1)+  f_{\tiny \mbox{BEP}}^{\delta=4}(y)\gamma_1+ O(\epsilon^3),
\end{align}
where $\gamma_1:=    \frac{3}{8}\,\epsilon^2\, \E[X^2]\, \frac{v'(u)^2}{v(u)^2}$, $\gamma_2:= \frac{1}{2 \sqrt{2}}\,\epsilon^2 \,\E[X^2]\, \frac{v''(u)}{v(u)}$ and
\begin{align}\label{BEPdelta}
f_{\tiny \mbox{BEP}}^{\delta}(y)=  \left(\left|\frac{y}{ \sqrt{2 v(u)}}\right|^{\delta} \rme^{-\frac{y^{2}}{2 v(u)}}\right) \, \left(\sqrt{2 v(u)} \,\,\Gamma\left(\frac{\delta+1}{2}\right)\right)^{-1}, \quad  y \in \R.
\end{align}
\end{Theorem}

\begin{proof}[Proof of Theorem \ref{theoTCLZ}]
The convergence in \eqref{TCLZ} is a direct consequence of Theorem \ref{quantitative}. In order to get the probability density function of $\Theta_\epsilon(u)$ in \eqref{eqhepsilon}, it is enough to compute $\E[\varphi(\Theta_\epsilon(u))]$ for any bounded positive function $\varphi$ as follows,
\begin{align*}
\E[\varphi(\Theta_\epsilon(u))] =\E\left[\E[\varphi(\Theta_\epsilon(u))|X]\right]&=\E\left[\int_{\R}\varphi(y)\,\phi(v(u-\epsilon X),y)\,\rmd y  \right]\\ &=\int_{\R}\varphi(y)\,\E[\phi(v(u-\epsilon X),y)]\,\rmd y,
\end{align*}
where Fubini-Tonelli theorem has been used for the last equality.
\\~\\
To get the  approximation of  $h_\epsilon$  in   \eqref{TaylorHepsilon}, we recall the following result that can be proved with Taylor expansion and easy algebra.

\begin{Lemma}\label{LemmaTaylor}
For any function $\varphi$ in $\mathcal C^2(\R)$ with bounded derivatives up to order two and any random variable $\eta$ with finite third moment,
\begin{equation*} \label{EPhiEta}
\E[\varphi(\eta)]=\varphi(\E\eta)+\frac 12\varphi''(\E\eta)\,\Var\eta+O(\E[|\eta-\E\eta|^3]),
\end{equation*}
where the constant in $O$-notation depends on $\Var\eta$ and on the bounds of derivatives of $\varphi$.
\end{Lemma}

Applying Lemma \ref{LemmaTaylor} with $\eta=\epsilon X$ and $\varphi(\cdot)=\phi(v(u-\cdot), y)$ for fixed $u,y$ and $\epsilon$, and bearing in mind that $\E[X]=0$, ones get
\[\E[\phi(v(u-\epsilon X),y)]=\phi(v(u),y)+\frac 12\varphi''(0)\,\epsilon^2\E[X^2]+O(\epsilon^3\E[|X|^3]),\]
where
\[\varphi''(0)=\partial^2_{vv}\phi(v(u),y)\,v'(u)^2+\partial_{v}\phi(v(u),y)\,v''(u). \]
Since
\begin{equation*}
\partial_{v}\phi(v, y) = \frac{\sqrt{\pi v} \rme^{-\frac{y^2}{2 v}} (y^2- v)}{2 \pi \sqrt{2} v^3} \quad  \mbox{ and } \quad   \partial^2_{vv}\phi(v, y)= \frac{\sqrt{\pi v}(y^4+3 v^2 -6 vy^2)\rme^{-\frac{y^2}{2 v} }}{4 v^5 \pi \sqrt{2}},
\end{equation*}
the proof is complete.
\end{proof}

\paragraph{Discussion on $h_\epsilon$ density.}

Since the density in \eqref{TaylorHepsilon} plays a crucial rule in our asymptotics and as its non-Gaussian shape was not previously studied in the literature,  in the following we propose an  analysis  of the truncated version of $h_\epsilon(y)$,  \emph{i.e.},

\begin{equation} \label{htilde}
\widetilde{h}_\epsilon(y) =  f_{\tiny \mbox{BEP}}^{\delta=0}(y) (1+ \gamma_1 - \gamma_2)+ f_{\tiny \mbox{BEP}}^{\delta=2}(y)(\gamma_2- 2 \gamma_1)+  f_{\tiny \mbox{BEP}}^{\delta=4}(y)\gamma_1,
\end{equation}
where $f_{\tiny \mbox{BEP}}^{\delta}(y)$ as in \eqref{BEPdelta} and $\gamma_1$, $\gamma_2$ as in  Theorem \ref{theoTCLZ}.  \\

Firstly, one can remark that coefficients $\gamma_1$ and $\gamma_2$ of the linear combination   $\widetilde{h}_\epsilon$ depend on the variance function $v(u)$  in \eqref{vu} and on its first and second derivatives.     For the nodal set with $u=0$ one can easily evaluate
$v(0)=  (2 \pi)^{-1} \int_{\R^2}  \arcsin(\rho(t)) \,\rmd t$. An illustration of  theoretical $u \mapsto v(u)$ in \eqref{vu},  $u \mapsto v^{'}(u)$    and $u \mapsto v^{''}(u)$  can be found in Figure \ref{fdeltafigure} (left panel).\\

Furthermore, notice that function  $f_{\tiny \mbox{BEP}}^{\delta}(y)$ in  \eqref{BEPdelta} is a particular case of the Bimodal Exponential Power density function,  \emph{i.e.}, $f_{\tiny \mbox{BEP}}(y)= (\alpha |\frac{y-\mu}{\zeta}|^{\delta} \rme^{- |\frac{y-\mu}{\zeta}|^{\alpha}}) \, (2 \zeta \Gamma(\frac{\delta+1}{\alpha}))^{-1}$, for $y \in \R$ (see \citet{Hassan})   with fixed values of parameters  $\alpha=2, \mu=0$ and $\zeta = \sqrt{2 v(u)}$ and  varying $\delta$. Obviously, for $y \in \R$, $f_{\tiny \mbox{BEP}}^{\delta=0}(y) = \phi(v(u), y)$, \emph{i.e.},  the Gaussian density with zero mean and variance  $v(u)$. An illustration of the behaviour of these bimodal densities for two different values of $u$ is given in Figure \ref{fdeltafigure} ($u=1.5$ center panel, $u=3$ right panel).\\

Theoretical resulting $\widetilde{h}_\epsilon$ functions in \eqref{htilde},  built by using   $v(u)$, $v'(u)$, $v''(u)$ and $f_{\tiny \mbox{BEP}}^{\delta}$ functions   studied above,   are displayed below in Section \ref{numerique} (see Figures \ref{ZepsilonFigure} and \ref{ZepsilonFigure2}). \\

\begin{figure}[H]
  \includegraphics[width=4.4cm,height=4.5cm]{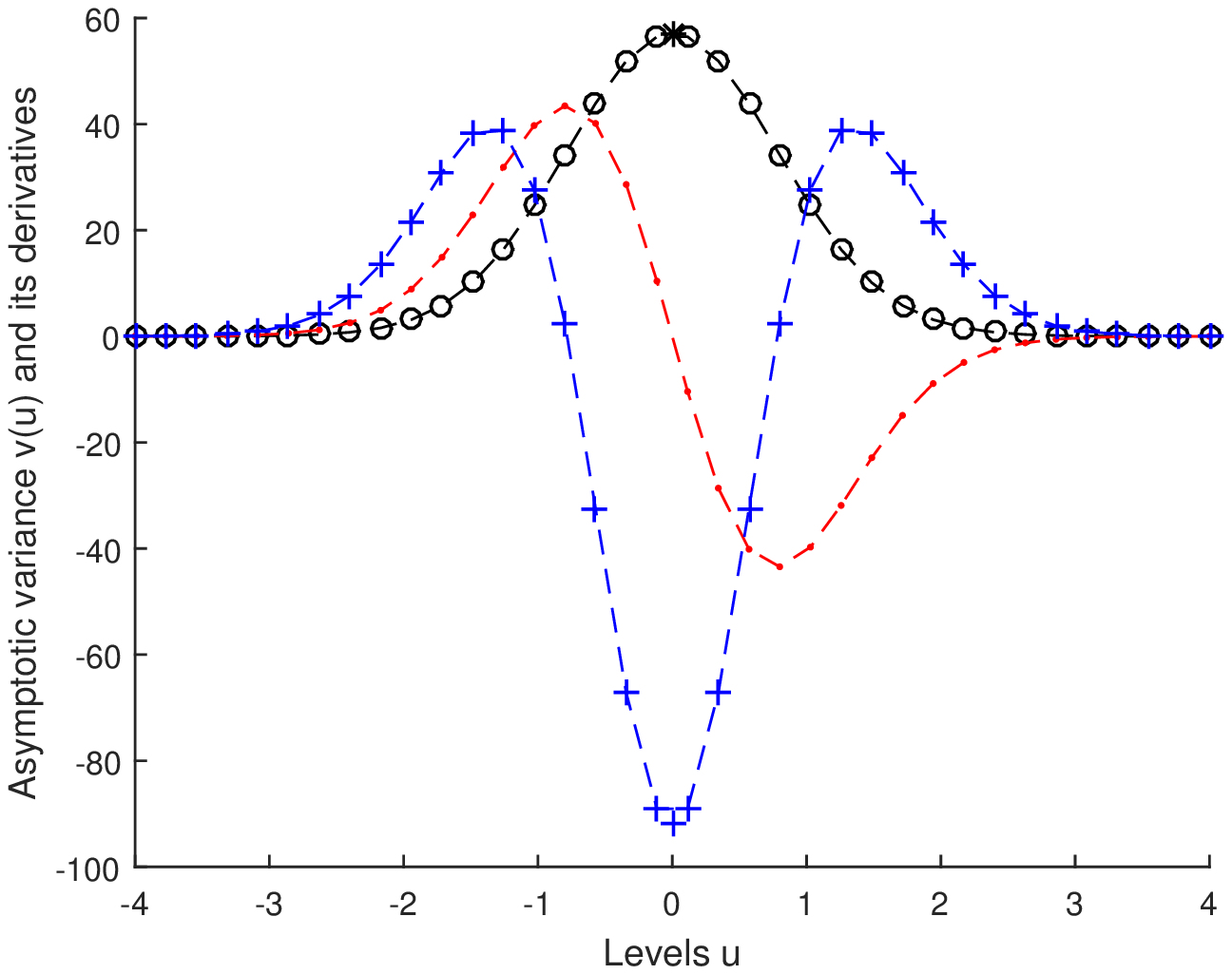}
 \includegraphics[width=4.4cm,height=4.5cm]{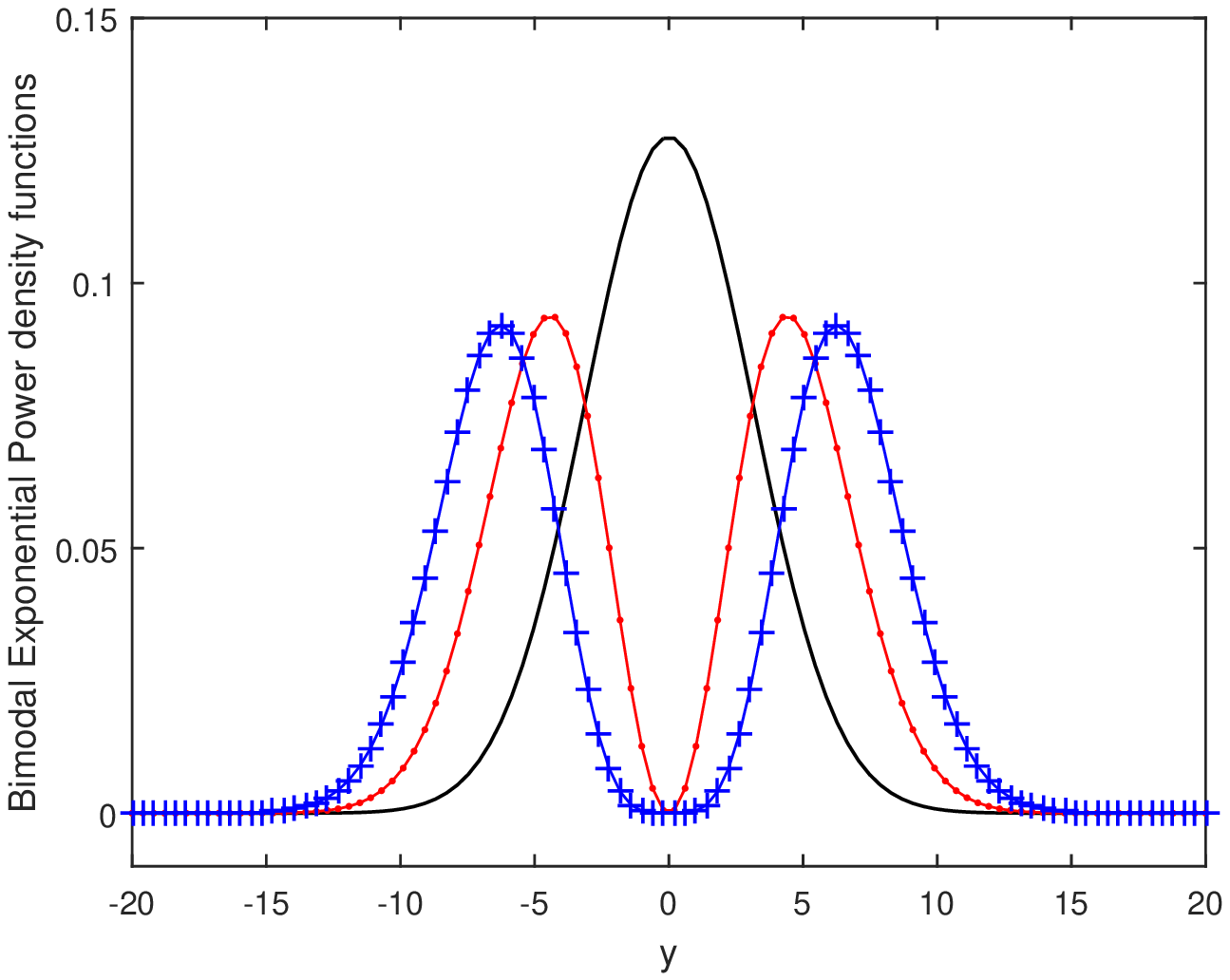}
   \includegraphics[width=4.4cm,height=4.5cm]{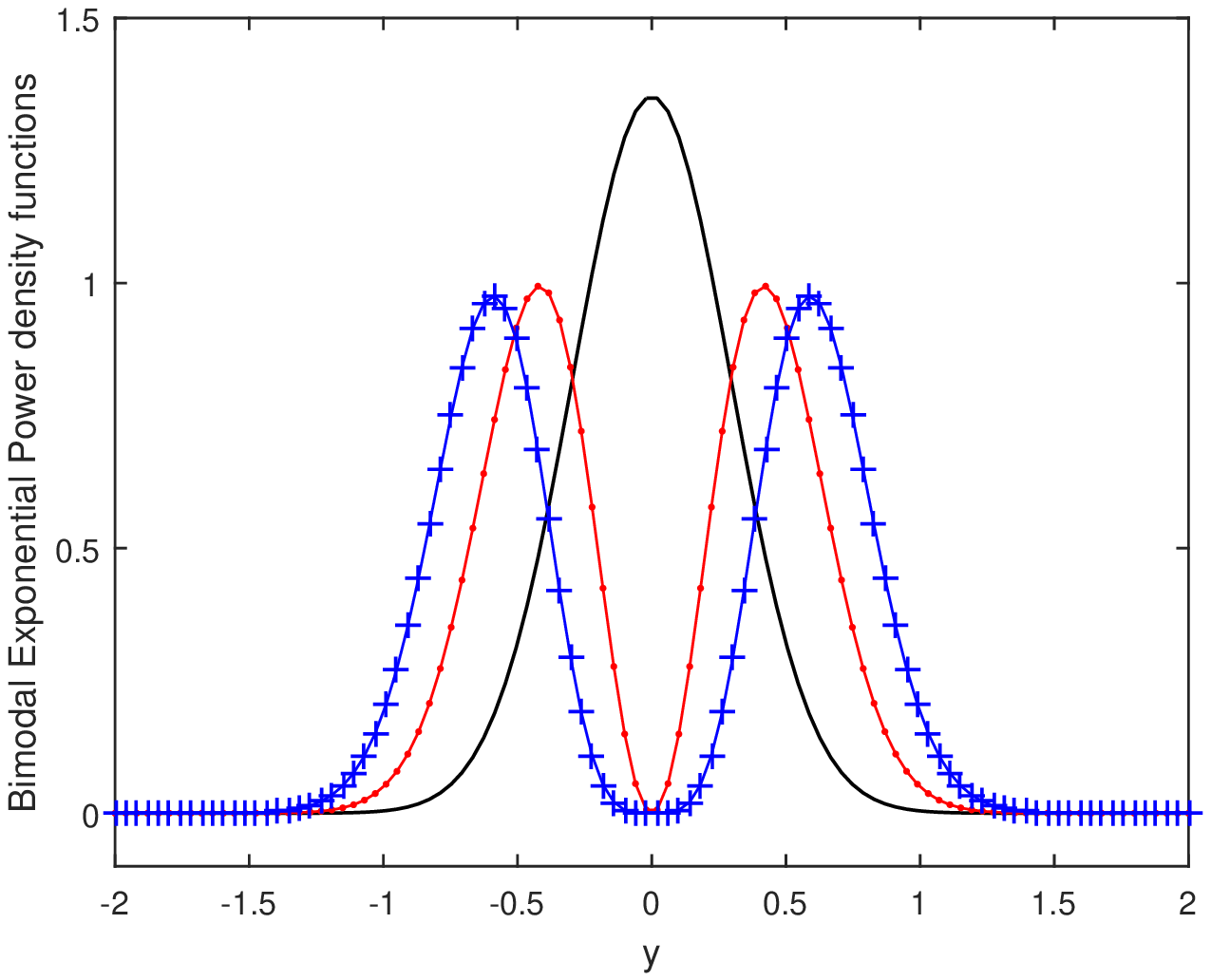}
  \vspace{-0.17cm}
 \caption{Left panel:   $u \mapsto v(u)$ in \eqref{vu} (dotted black line), $u \mapsto v^{'}(u)$  (dashed  red line) and $u \mapsto v^{''}(u)$ (crossed blue line), for several levels $u$. Black point represents the   value  $v(0)$. Center and right panels:  theoretical  Bimodal Exponential Power density functions   for $u=1.5$ (center panel) and  $u=3$ (right panel). We display  $f_{\tiny \mbox{BEP}}^{\delta=0}$ (black line)  $f_{\tiny \mbox{BEP}}^{\delta=2}$ (red line) $f_{\tiny \mbox{BEP}}^{\delta=4}$ (blue  line), with  $f_{\tiny \mbox{BEP}}^{\delta}$  as in \eqref{BEPdelta}  with $v(u)$ as in \eqref{vu}.    The considered correlation function is    $\rho(t)=  \rme^{-\kappa^2 \parallel t \parallel^2}$, with   $\sigma_g = 1$ and $\kappa = 100/2^{10}$.  } \label{fdeltafigure}
  \end{figure}

\subsection{Asymptotics for $\epsilon \to 0$ and  $T\nearrow\R^2$}\label{CLTarea}

Let $T^{(N)}=NT$, as introduced in Section \ref{preliminary}. In the following we prove that  $Y^{\epsilon_N}_{T^{(N)}}$  given by \eqref{YTu}  satisfies a classical Central Limit Theorem as soon as $\epsilon_N$ goes to 0 sufficiently fast,  for $N \rightarrow\infty$.

\begin{Theorem}\label{theoremThetaAsym}
Let  $f(t) = g(t) + \epsilon \, X$,   $t\in \R^2$ as in Definition \ref{MODELadditive} and  $\epsilon_N$  be such that
\begin{equation} \label{cond_epsilon}
\lim_{N \rightarrow\infty}N\,\epsilon_N=0.
\end{equation}
Then it holds that,
\begin{align*}
 Y^{\epsilon_N}_{T^{(N)}}(u)= |T^{(N)}|^{1/2}\left(C^{/T^{(N)}}_2(f,u)-\E[C^{/T^{(N)}}_2(f,u)] \right) \overset{d}{\underset{N\to \infty}{\longrightarrow}} \mathcal N(0,v(u)),
\end{align*}
with $v(u)$ given by \eqref{vu}.
\end{Theorem}

\begin{proof}[Proof of Theorem \ref{theoremThetaAsym}] We start by writing $Y^{\epsilon}_{T}(u)=(Y^\epsilon_T(u)-\Theta_\epsilon(u))+\Theta_\epsilon(u)$.\smallskip

On the one hand, by triangular inequality we have
\begin{equation}\label{eqTr}
d_W(Y^{\epsilon}_{T}(u),\Theta_\epsilon(u)) \le d_W(Y^\epsilon_T(u), Z_T^\epsilon(u)) + d_W(Z^\epsilon_T(u), \Theta_\epsilon(u))
\end{equation}
From \eqref{Y=Z+R}, we have  $d_W(Y^\epsilon_T(u), Z_T^\epsilon(u)) \le \sqrt{\mathbb E[R_T^\epsilon(u)^2]}.$
Then, since
\begin{equation*}
\E[R^{\epsilon}_T(u)^2] =|T| \E[(C^*_2(g,u-\epsilon X)-\E[C^*_2(g,u-\epsilon X)])^2]
\end{equation*}
and from  \eqref{C2Tadditive},   $\E[C^*_2(g,u-\epsilon X)] = \Psi(u/\sigma_g)+ \frac{\epsilon^2}{\sigma_g^2}
\frac{\Psi^{''}(u/\sigma_g)}{2} \E[X^2]  + O\left (\frac{\epsilon^3}{\sigma_g^3} \E[|X|^3]\right)$,
one can get
\begin{equation*}
\mathbb E[R^{\epsilon}_T(u)^2]= |T| \left ( \Psi^{'}\left(\frac{u}{\sigma_g}\right)^2 \frac{\epsilon^2 \E[X^2]}{\sigma_g^2} + O(\epsilon^3)\right ) = \epsilon^2 |T| \left (\kappa_1\E[X^2]+O\left (\epsilon \right ) \right ),
\end{equation*}
where $\kappa_1 >0$ and the constant involved in the $O$-notation depends neither on $\epsilon$ nor on $T$. Then, from condition in \eqref{cond_epsilon}, the first term on the r.h.s. of \eqref{eqTr} with $\epsilon=\epsilon_N$ and $T=T^{(N)}$ goes to 0 as $N$ goes to infinity.\\
Concerning the second term, Theorem \ref{quantitative} yields $\kappa_2\,(\log|T|)^{-1/12}$ as upper bound, where $\kappa_2$ does not depend on $\epsilon$. Therefore, the second term on the r.h.s. of \eqref{eqTr} goes to 0 as $T\nearrow \R^2$ uniformly with respect to $\epsilon$ (see Theorem  \ref{theoTCLZ}).\\
Finally, thanks to the Wasserstein distance in \eqref{eqTr} that goes to 0, we get that $Y^{\epsilon_N}_{T^{(N)}}(u)-\Theta_{\epsilon_N}(u)$ converges to 0 in distribution. \smallskip

On the other hand, $\Theta_{\epsilon_N}(u)\overset{d}{\underset{N\to \infty}{\rightarrow}} \mathcal N(0,v(u))$ since $h_\epsilon(y)\to \phi(v(u),y)$ as $\epsilon \to 0$. At last, Slutsky theorem allows us to conclude.
\end{proof}

\subsection{Numerical illustrations}\label{numerique}

All over this section, $\sigma_g$ is assumed to be equal to 1. In the following,  by using histograms we compare  the empirical density of the random variable  $Z^{\epsilon}_T(u) :=  |T|^{1/2} (C_2^{/T}(f,u)-\Psi(u - \epsilon X))$ versus the truncated  probability density function of $\Theta_{\epsilon}$, \emph{i.e.}, $\widetilde{h}_\epsilon$ given in \eqref{htilde}. Each  histogram is built by reproducing  300 Montecarlo independent simulations in a large domain such that  $|T| = 1024^2$. \smallskip

\paragraph{Case 1: $X$ is Skellam distributed.}

Firstly,  we consider the case where $X$ follows a  discrete  Skellam probability distribution which is the difference  of two   independent  Poisson-distributed random variables with respective expected values $\mu_1$   and $\mu_2$. We  choose  the  parameters setting gathered in  Table \ref{parametersetting}.  Obtained results are shown Figure \ref{ZepsilonFigure}   for $u=1.5$ (first row)  and  for $u=3$ (second row). Furthermore,
necessary preliminary studies to built $\widetilde{h}_\epsilon$ as in \eqref{htilde}, on  BEP functions,  $u \mapsto v(u)$ and its derivatives are given in  Section  \ref{asymtoticSmallEpsilon}.

{\scriptsize
\begin{table}[H]
\centering
        \begin{tabular}{c|c|c|c|c|c|c|c}
            \hline \hline
        {\Large\strut}     $u$ & $X$ &     $\epsilon$     &      $\varepsilon : = \epsilon^2\, \E[X^2]$        &     $\gamma_1$      &      $\gamma_2$        &     Figure \ref{ZepsilonFigure}   &   \\ \hline
       {\large\strut}     \multirow{3}{*}{$1.5$} & \multirow{3}{*}{Skellam$_{\mu_1=\mu_2=1}$} & 0.5 & 0.5  & 0.979  & 0.686 & left panel    &  \multirow{3}{*}{first row} \\ \cline{3-7}
         {\large\strut}               & &  0.3  &  0.18   &  0.352   &  0.245 & center  panel &  \\ \cline{3-7}
           {\large\strut}            &   &  0.1  & 0.02   &  0.039   &  0.028 & right  panel & \\     \hline \hline
           {\large\strut}                \multirow{3}{*}{$3$} & \multirow{3}{*}{Skellam$_{\mu_1=\mu_2=1}$}  & 0.5 & 0.5  &  2.818 & 2.508 & left panel & \multirow{3}{*}{second row} \\ \cline{3-7}
         {\large\strut}               & &  0.3  &  0.18   & 1.015    & 0.903   & center  panel &\\ \cline{3-7}
           {\large\strut}            &   &  0.1  & 0.02   & 0.113     & 0.101  & right  panel & \\     \hline \hline
        \end{tabular}   \vspace{0.3cm}
\caption{Parameters setting associated to Figure \ref{ZepsilonFigure}. Here  $|T| = 1024^2$,  $\mu=0$, $\sigma_g=1$  and  $\rho(t)= e^{-\kappa^2 \parallel t \parallel^2}$,   for $\kappa = 100/2^{10}$, \emph{i.e.}, $\lambda= 0.019$.   }\label{parametersetting}
\end{table}}

\begin{figure}[H] 
 \includegraphics[width=4.4cm, height=4.5cm]{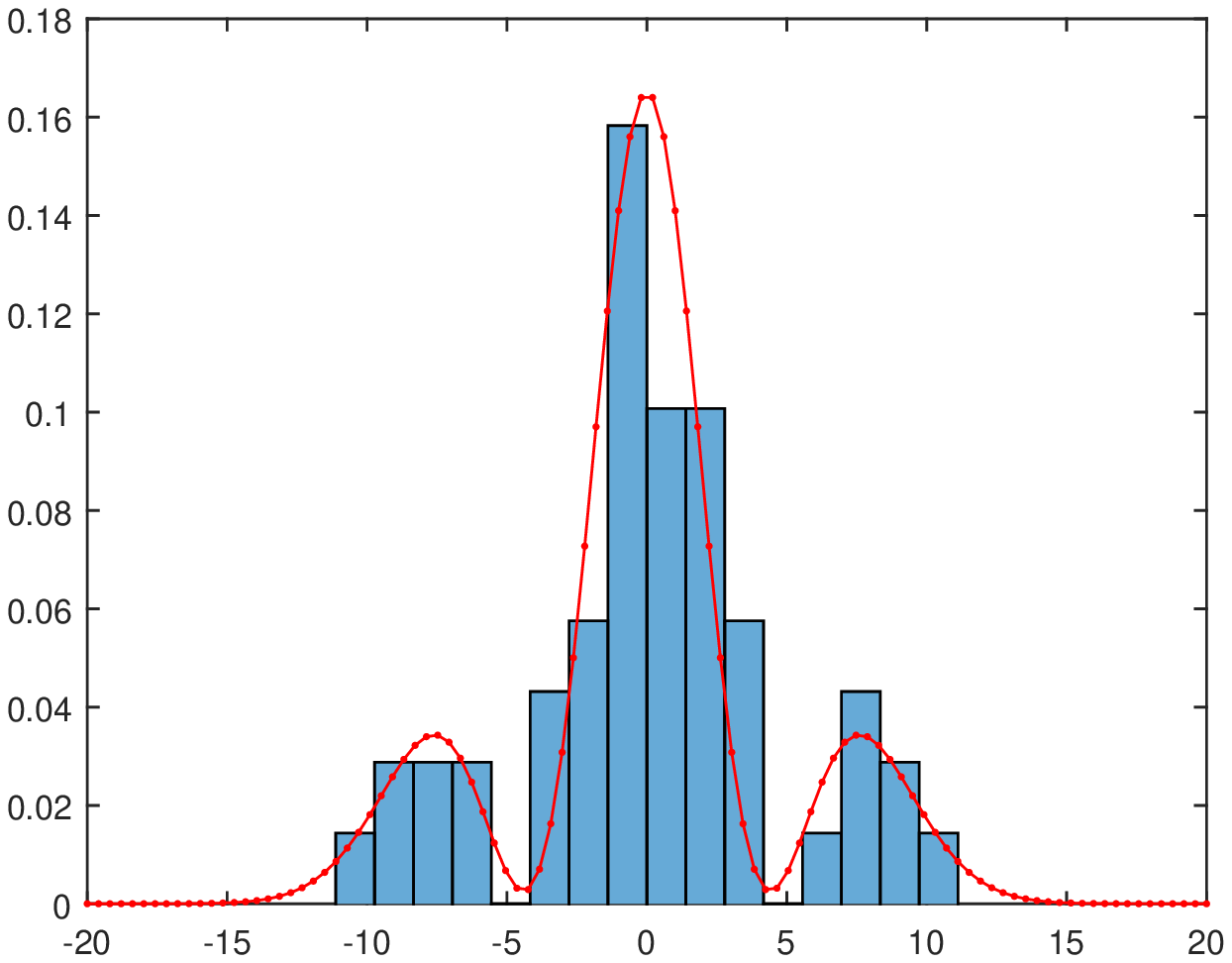}
 \includegraphics[width=4.4cm, height=4.5cm]{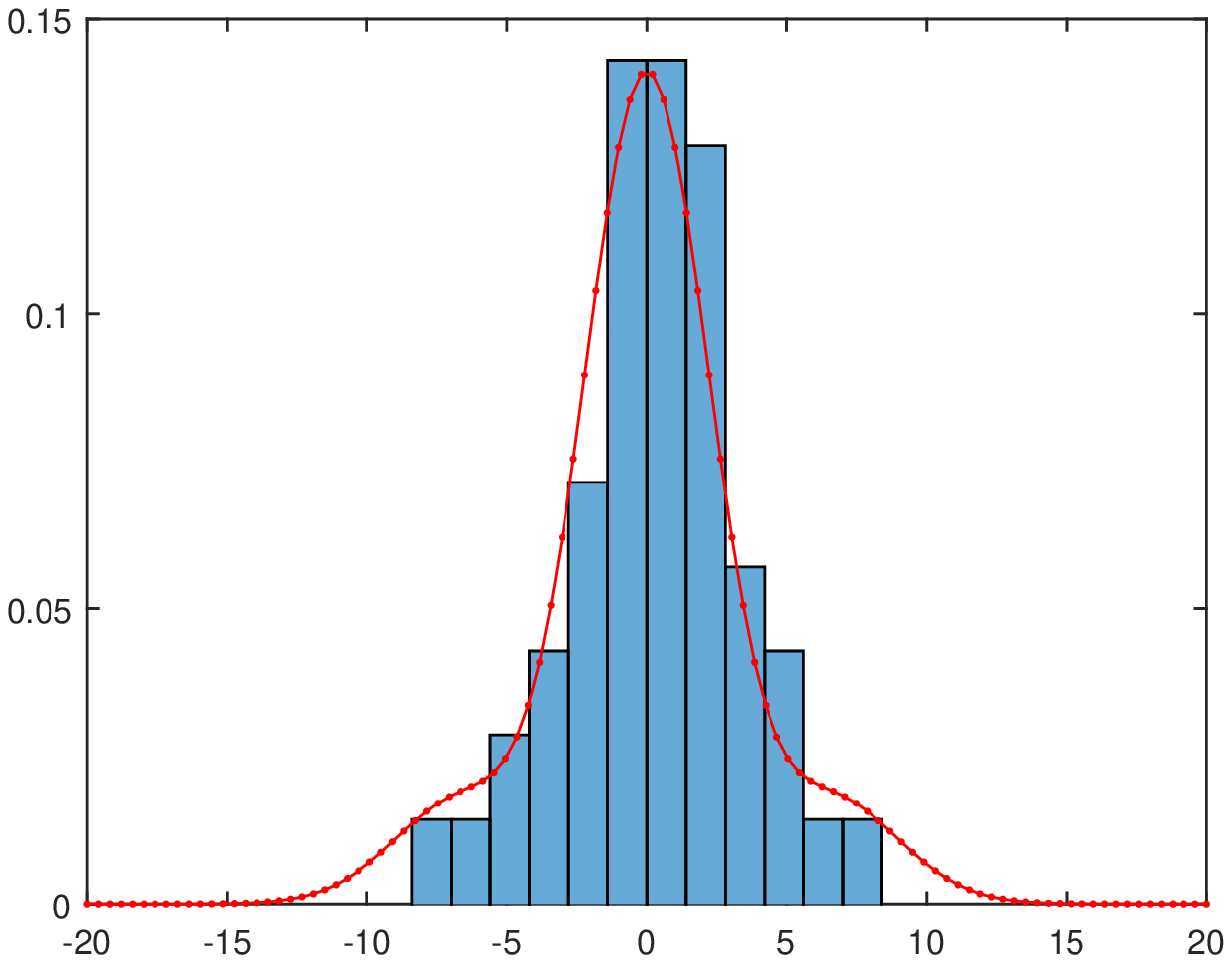}
 \includegraphics[width=4.4cm, height=4.5cm]{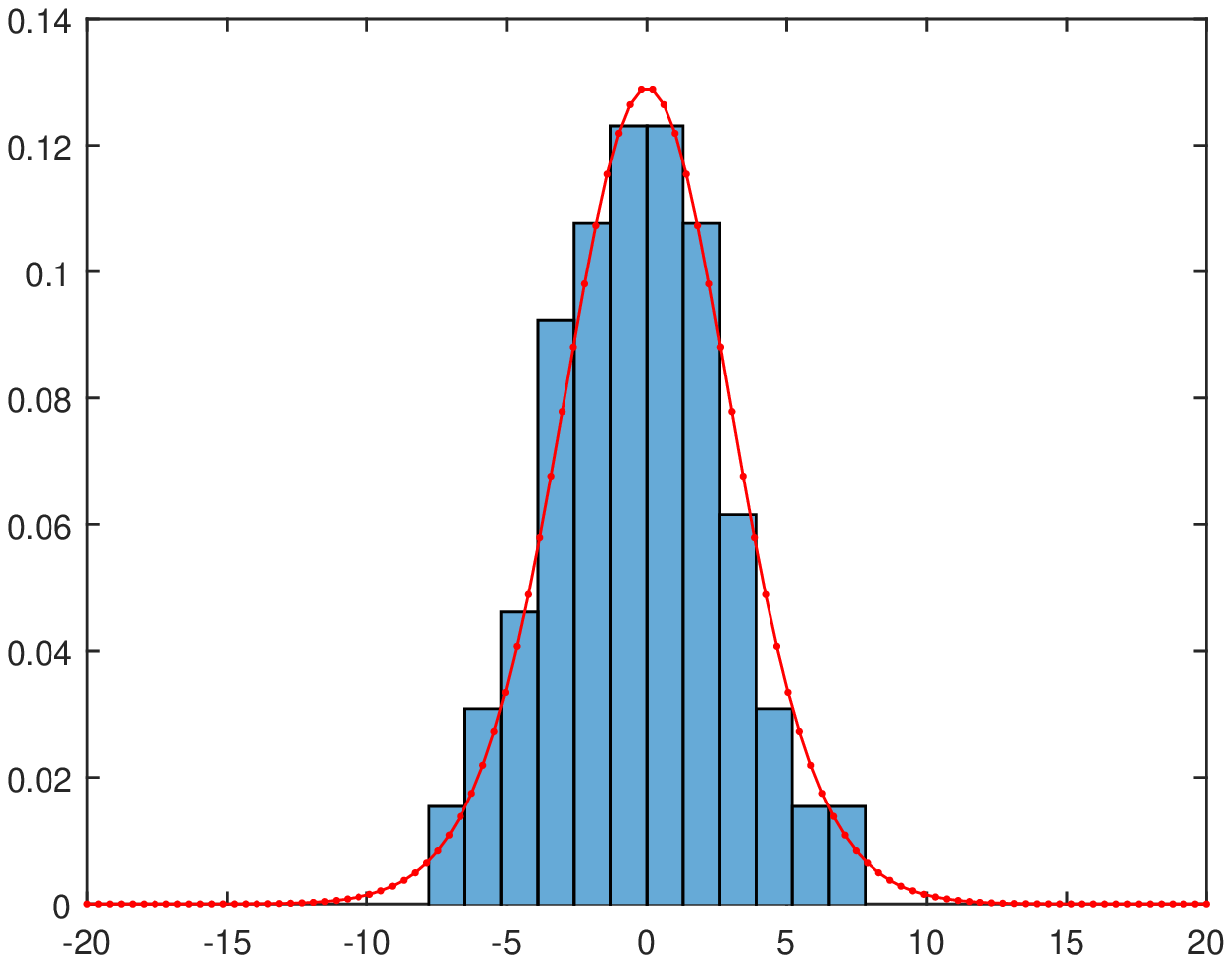}\\
 \includegraphics[width=4.4cm, height=4.5cm]{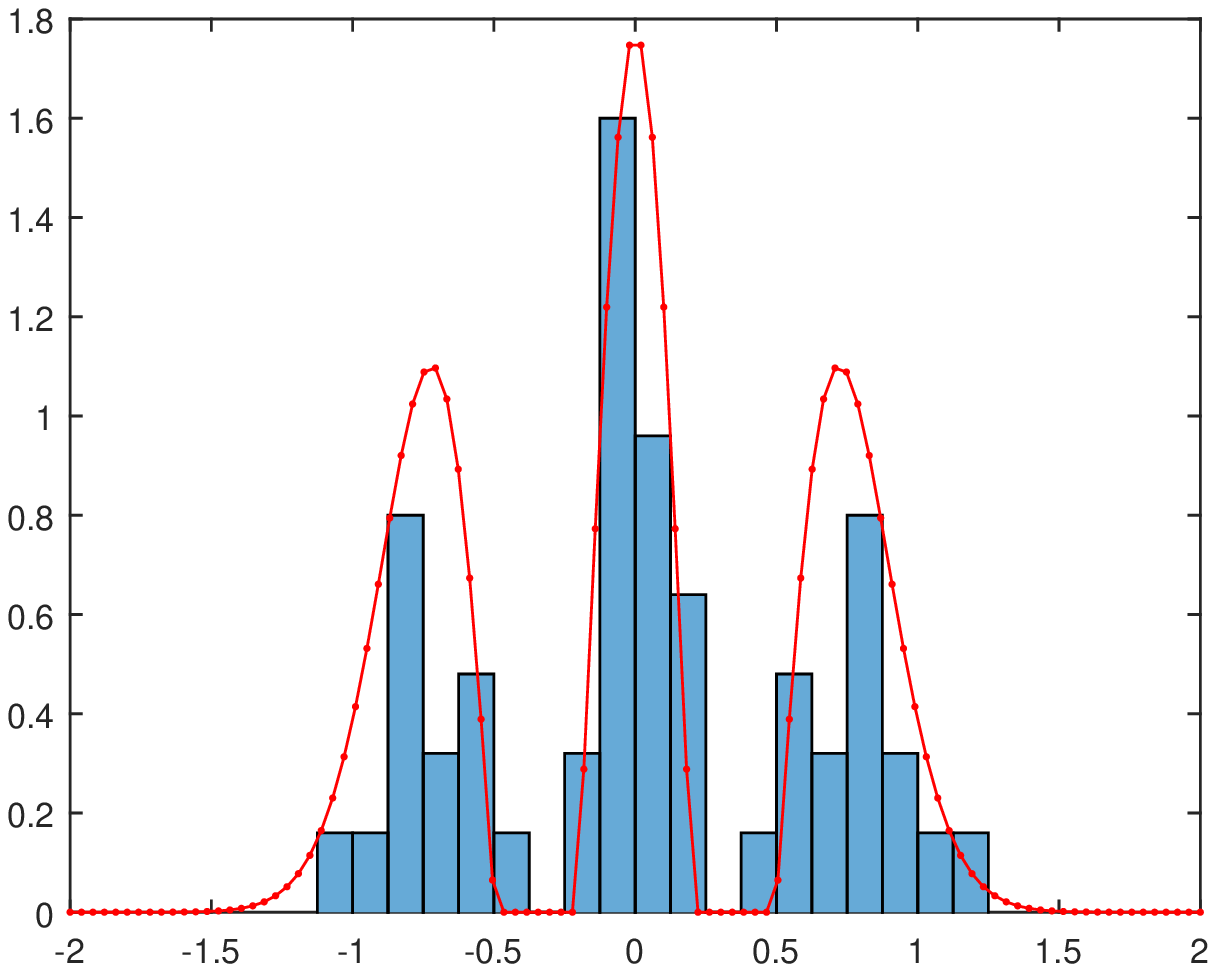}
 \includegraphics[width=4.4cm, height=4.5cm]{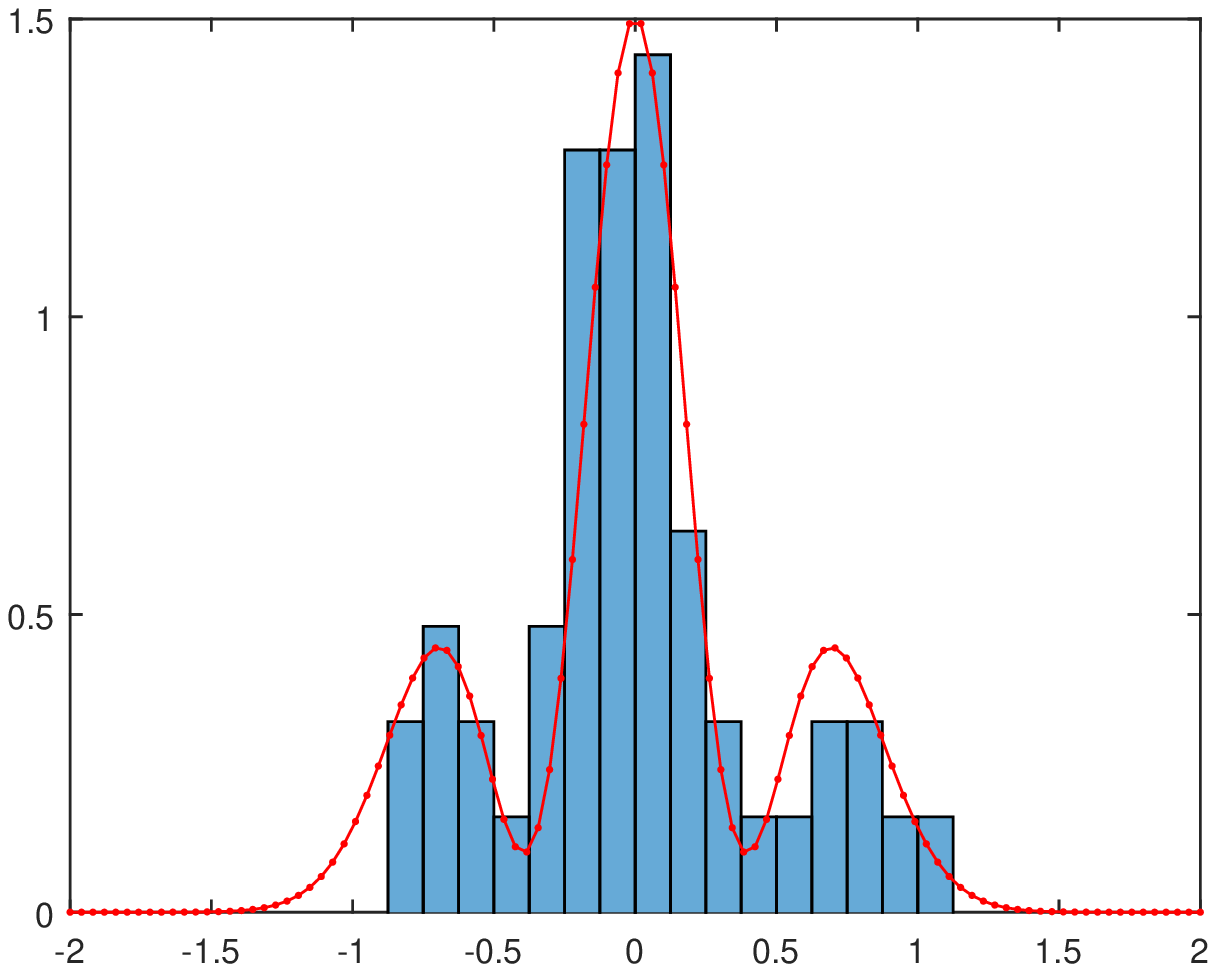}
 \includegraphics[width=4.4cm, height=4.5cm]{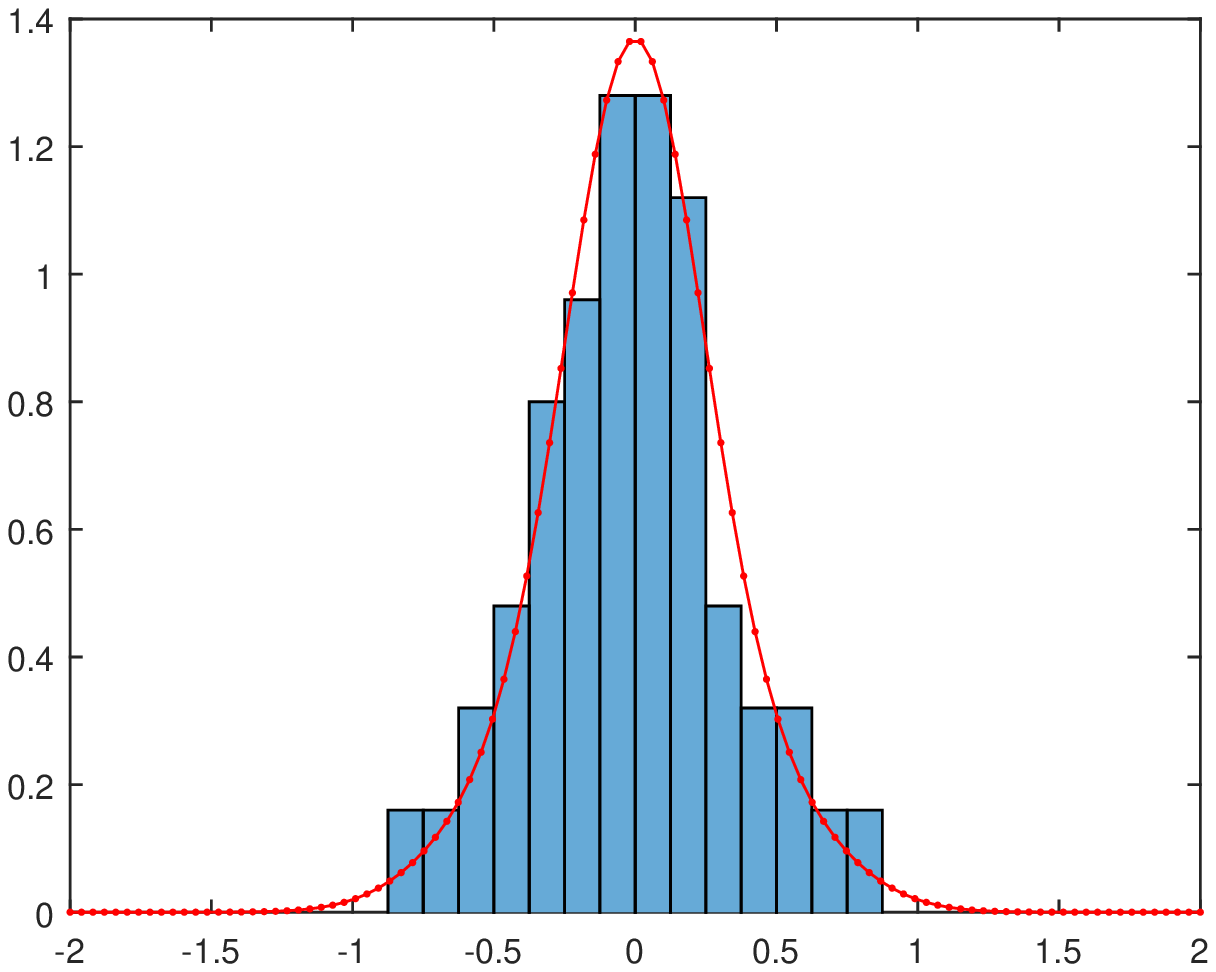}
    \vspace{-0.17cm}
   \caption{\textbf{Histogram for the study of density of $Z^{\epsilon}_T$ when  $X$ is Skellam distributed}, for $u=1.5$ (first row) and  $u=3$ (second row),  based on 300 Montecarlo independent simulations.  The chosen parameters setting is gathered  in Table \ref{parametersetting}.    Necessary preliminary studies to build $\widetilde{h}_\epsilon$ as in \eqref{htilde}, on  BEP functions,  $u \mapsto v(u)$ and its derivatives are given in Figure \ref{fdeltafigure}. Resulting theoretical $\widetilde{h}_\epsilon$ density is drawn by using red plain line. }\label{ZepsilonFigure}
\end{figure}

\paragraph{Case 2: $X$ is $t$-distributed.}

We now consider the case where $X$ follows a $t$-distribution  and  the  parameters are those  in  Table \ref{parametersetting2}.  Obtained results are shown Figure \ref{ZepsilonFigure2}   for $u=1.5$ (first row)  and  for $u=3$ (second row). Preliminary studies of BEP functions,  $u \mapsto v(u)$  and its derivatives are identical to those in  Section \ref{asymtoticSmallEpsilon}.

{\scriptsize
\begin{table}[H]
\centering
        \begin{tabular}{c|c|c|c|c|c|c|c}
            \hline \hline
        {\Large\strut}     $u$ & $X$ &    $\epsilon$     &      $\varepsilon : = \epsilon^2\, \E[X^2]$        &     $\gamma_1$      &      $\gamma_2$        &     Figure \ref{ZepsilonFigure2}   &   \\ \hline
       {\large\strut}     \multirow{3}{*}{$1.5$} & \multirow{3}{*}{$t_{\nu=5}$} & 0.5 &  0.417    &  0.816  &  0.576 & left panel    &  \multirow{3}{*}{first row} \\ \cline{3-7}
         {\large\strut}               & &  0.3  &   0.150 &    0.294    & 0.206  & center  panel &  \\ \cline{3-7}
           {\large\strut}            &   &  0.1  &  0.017   &   0.033   &  0.023 & right  panel & \\     \hline \hline
           {\large\strut}                \multirow{3}{*}{$3$} & \multirow{3}{*}{$t_{\nu=5}$} & 0.5 & 0.417    &   2.349   &   2.091  & left panel & \multirow{3}{*}{second row} \\ \cline{3-7}
         {\large\strut}               &  &  0.3   &  0.150    &  0.846   &   0.753  & center  panel &\\ \cline{3-7}
           {\large\strut}            &   &  0.1   & 0.017    &   0.094   & 0.084  & right  panel & \\     \hline \hline
        \end{tabular}    \vspace{0.3cm}
\caption{Parameters setting associated to Figure \ref{ZepsilonFigure2}. Here  $|T| = 1024^2$,  $\mu=0$, $\sigma_g=1$  and  $\rho(t)= e^{-\kappa^2 \parallel t \parallel^2}$,   for $\kappa = 100/2^{10}$, \emph{i.e.}, $\lambda= 0.019$. }\label{parametersetting2}
\end{table}}

\begin{figure}[H] 
\includegraphics[width=4.4cm,  height=4.5cm]{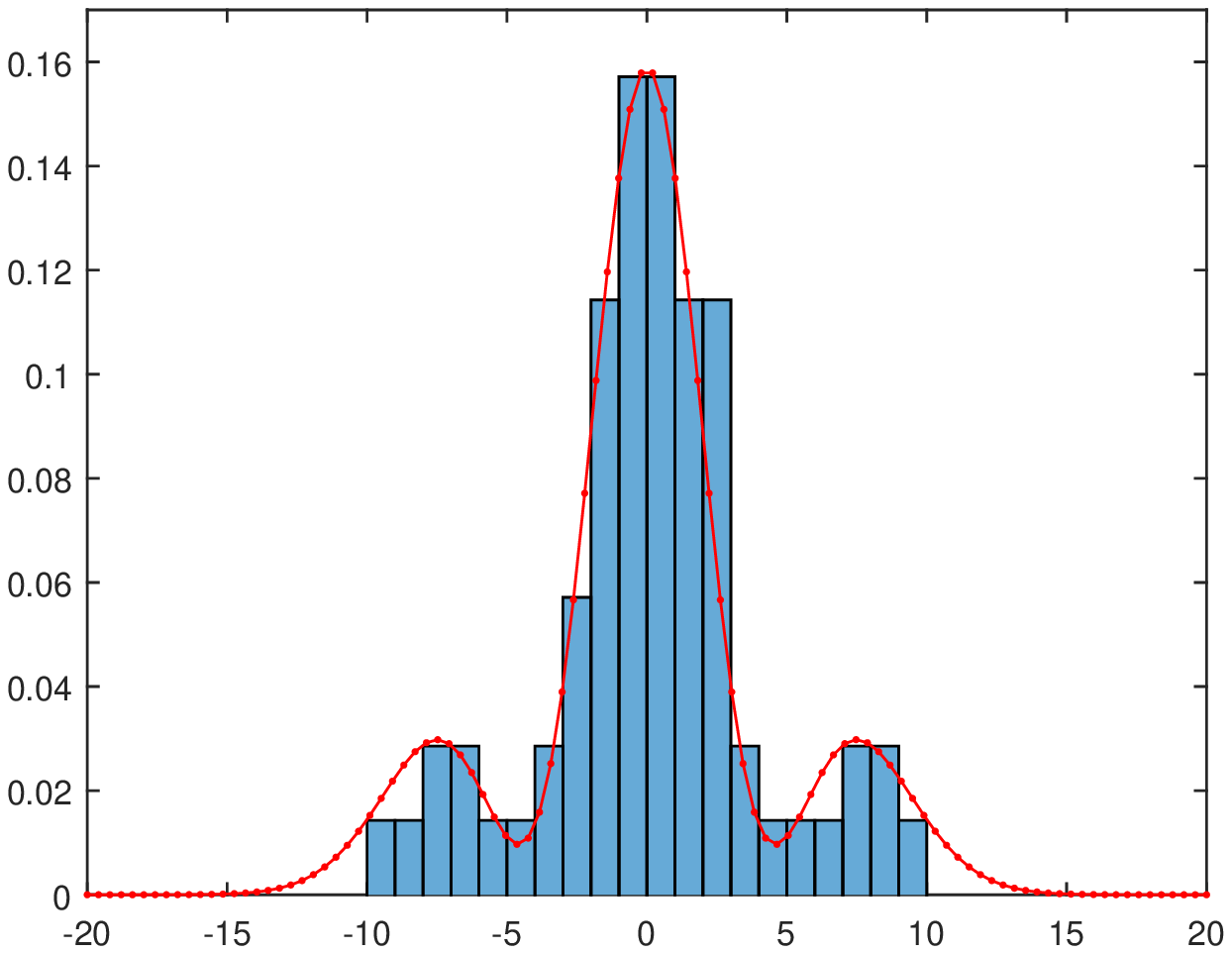}
\includegraphics[width=4.4cm, height=4.5cm]{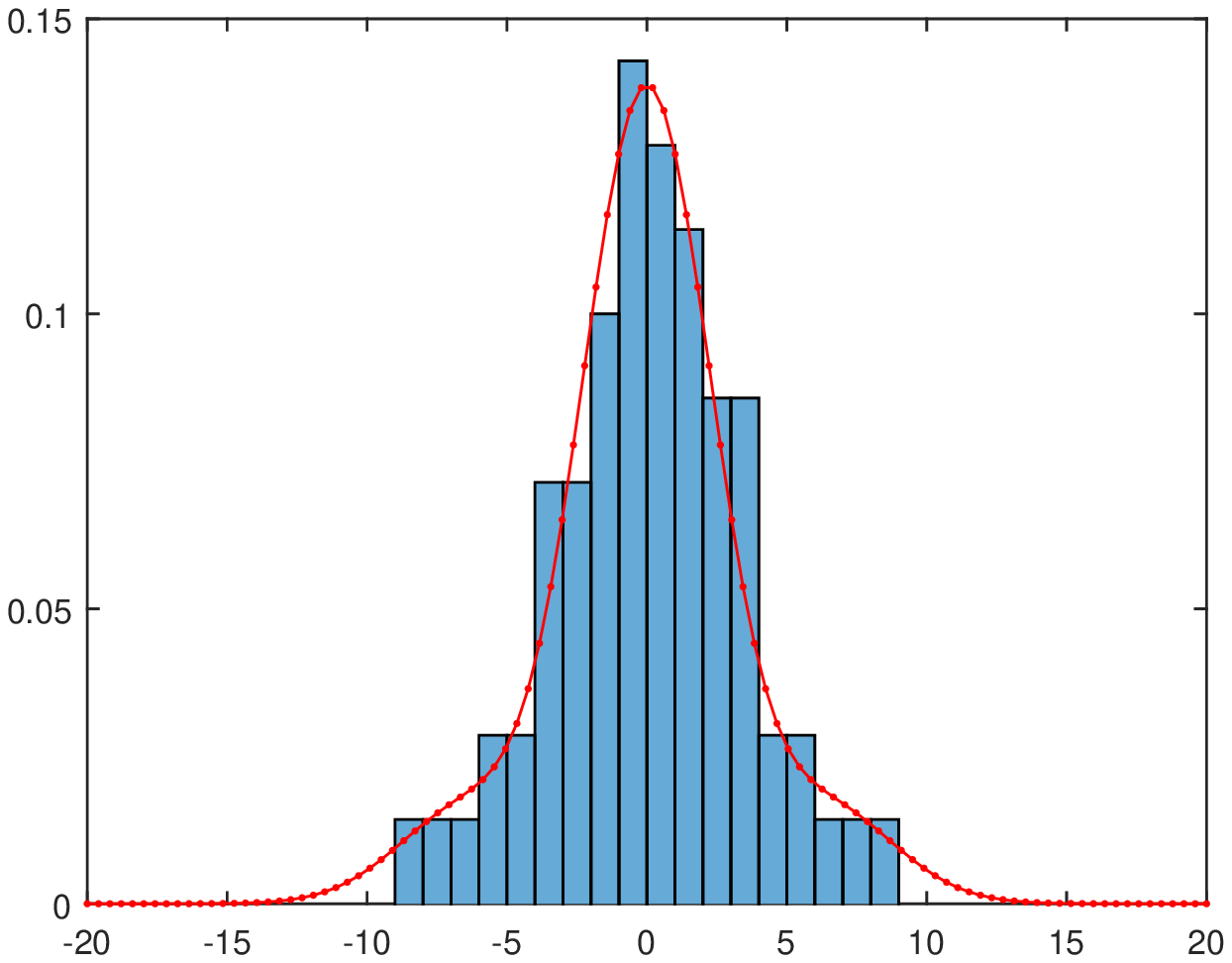}
\includegraphics[width=4.4cm, height=4.5cm]{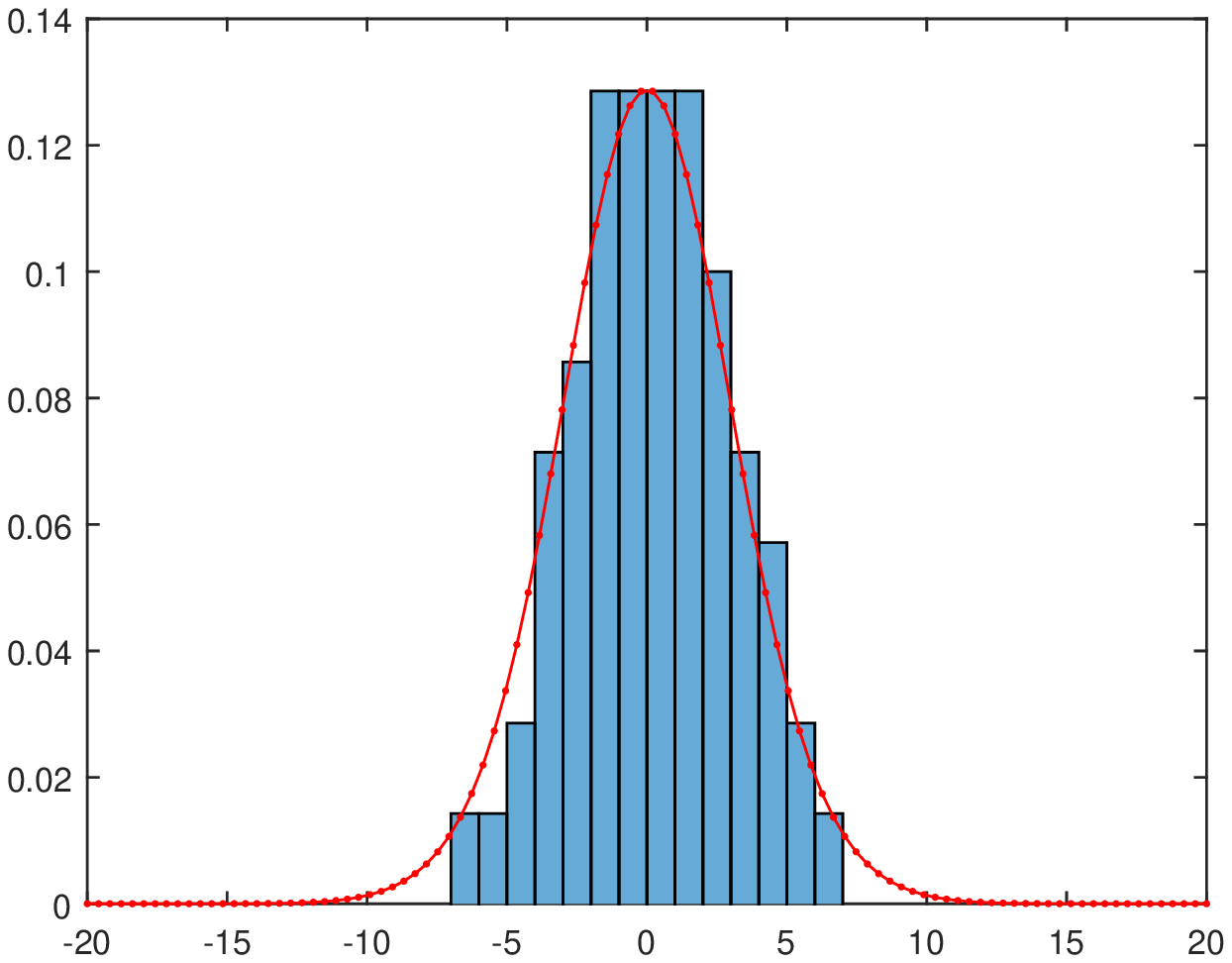}\\
\includegraphics[width=4.4cm, height=4.5cm]{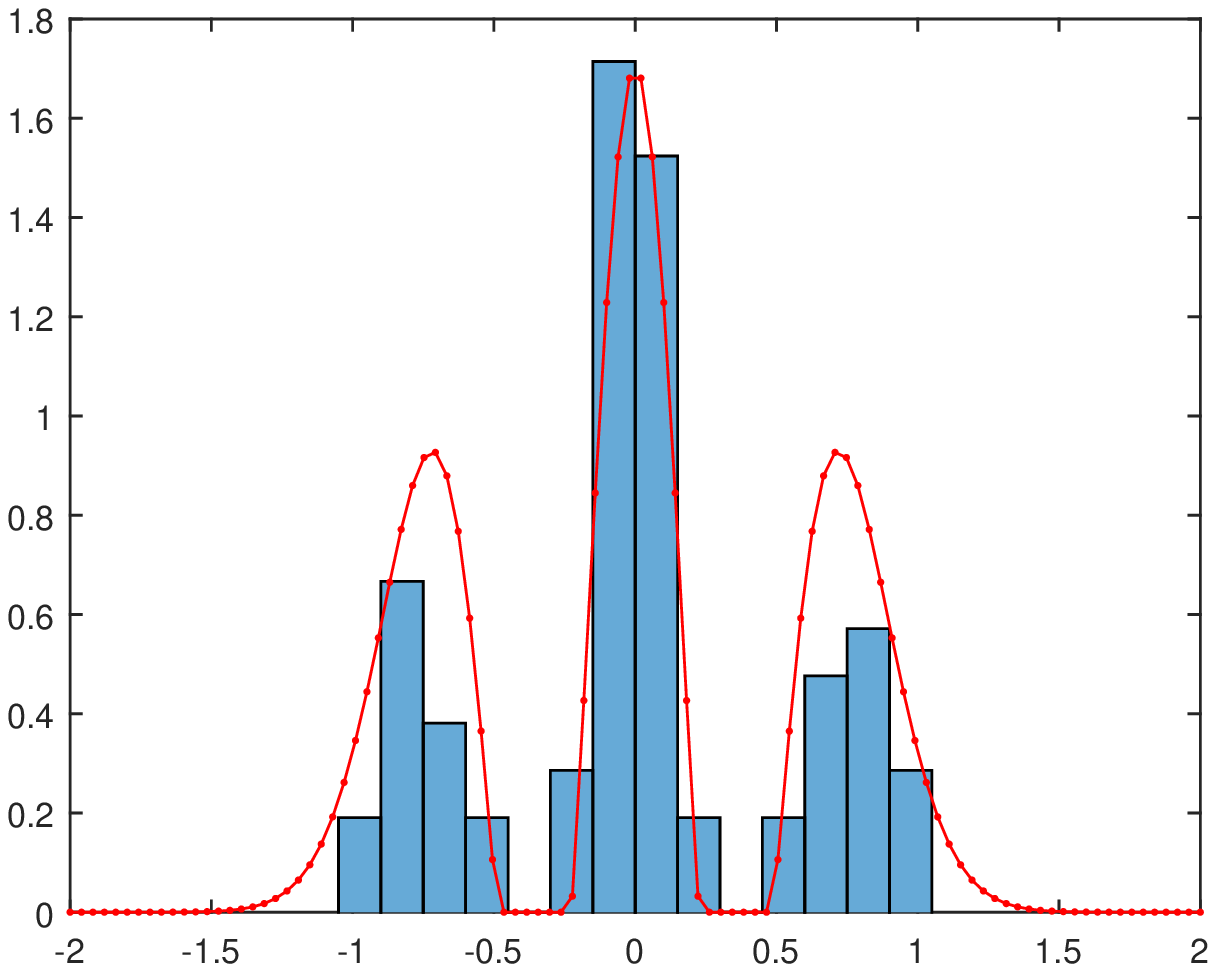}
\includegraphics[width=4.4cm, height=4.5cm]{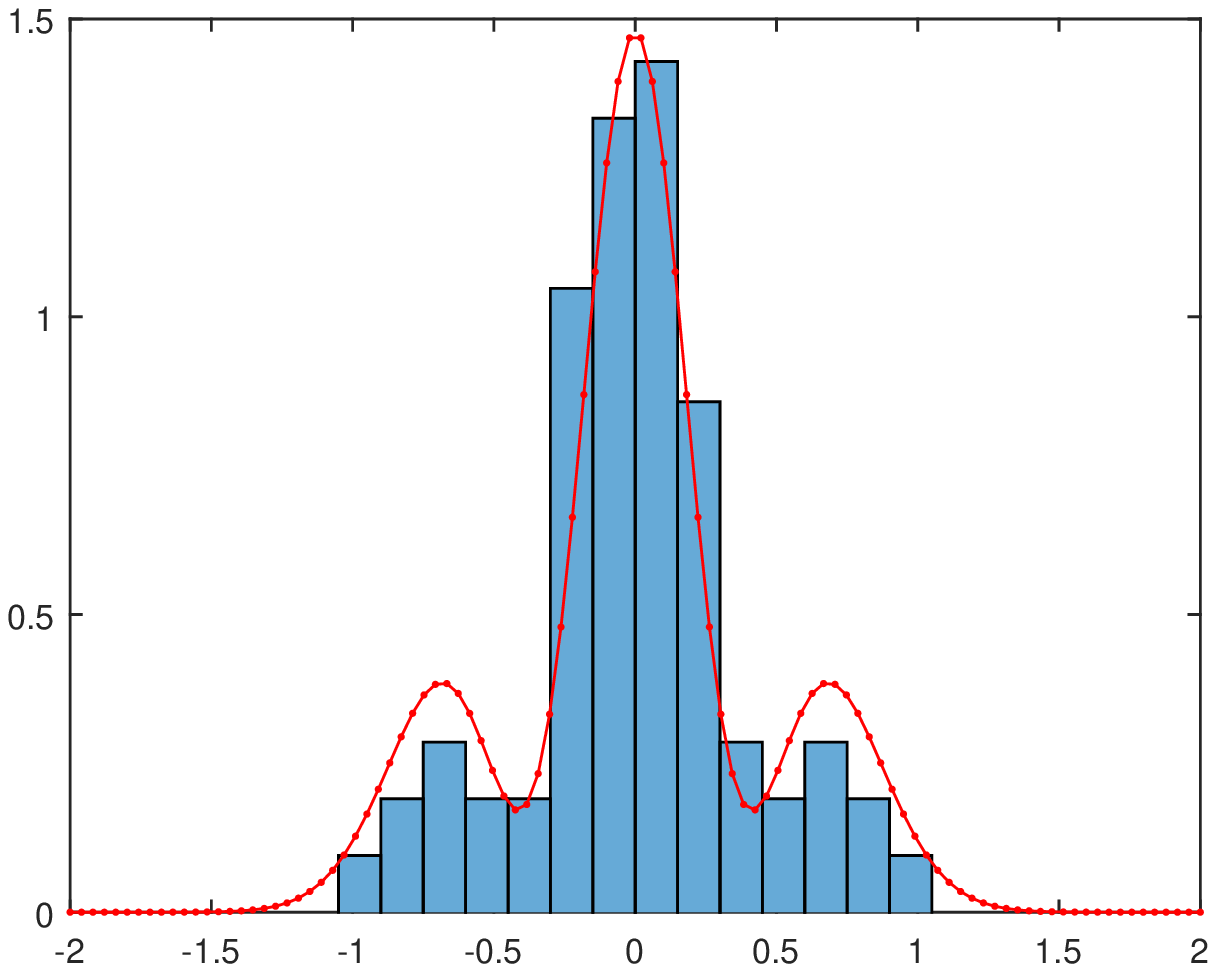}
\includegraphics[width=4.4cm, height=4.5cm]{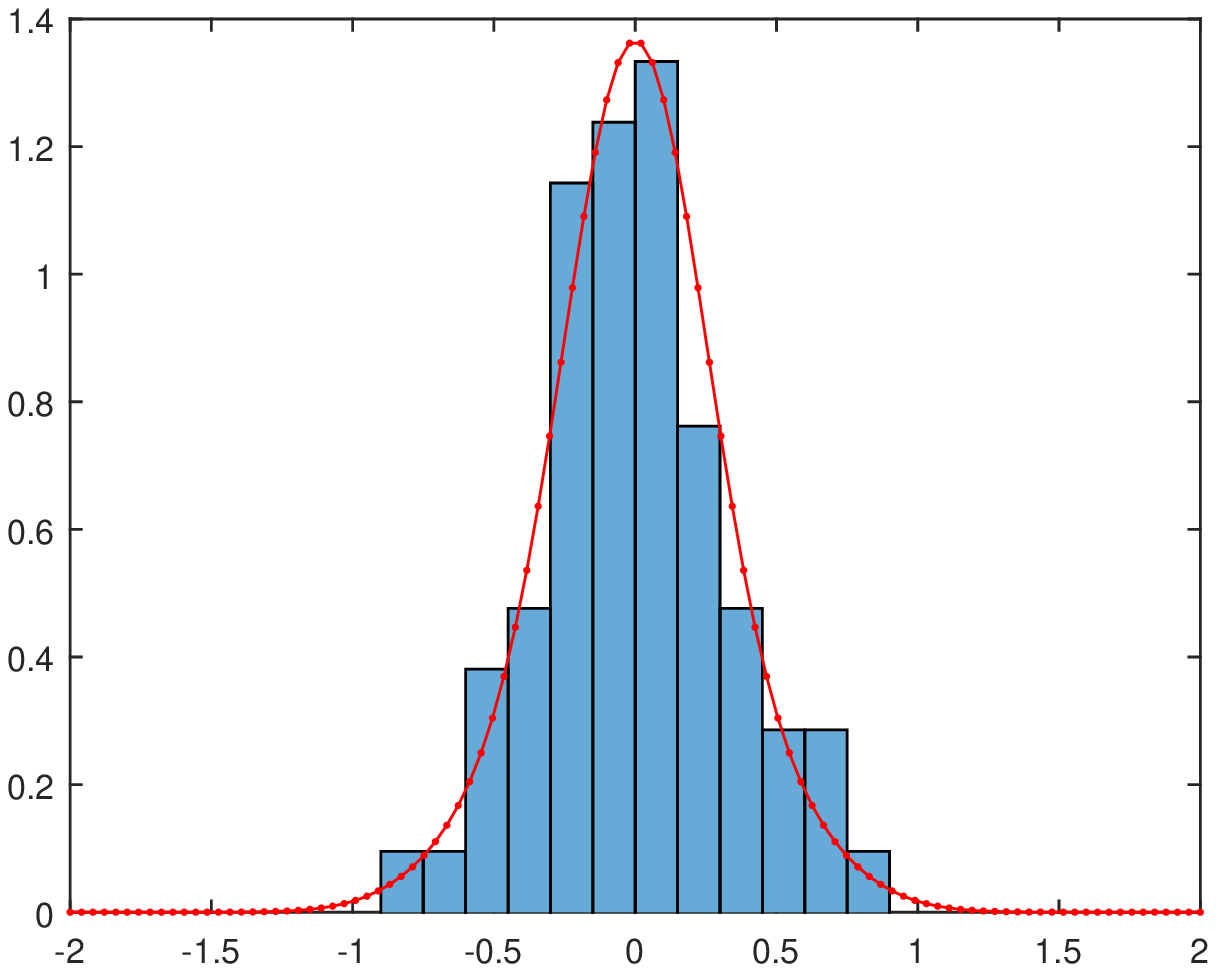}
    \vspace{-0.17cm}
   \caption{\textbf{Histogram for the study of density of $Z^{\epsilon}_T$ when $X$ is $t$-distributed}, for $u=1.5$ (first row) and  $u=3$ (second row),  based on 300 Montecarlo independent simulations.    The chosen parameters setting is gathered  in Table \ref{parametersetting2}.   Necessary preliminary studies to build $\widetilde{h}_\epsilon$ as in \eqref{htilde}, on  BEP functions,  $u \mapsto v(u)$ and its derivatives are given in Figure \ref{fdeltafigure}. Resulting theoretical $\widetilde{h}_\epsilon$ density is drawn by using red plain line.  }\label{ZepsilonFigure2}
\end{figure}

The bimodal behaviour of $\widetilde{h}_\epsilon$ in \eqref{htilde} is clearly visible  in Figures \ref{ZepsilonFigure}-\ref{ZepsilonFigure2}.
Furthermore in the numerical studies above one can appreciate the contiguity property of the proposed model for $\epsilon \to 0$. Indeed since theoretically $h_\epsilon(y)\to \phi(v(u),y)$ as $\epsilon \to 0$, in Figures \ref{ZepsilonFigure}-\ref{ZepsilonFigure2} the unimodal Gaussian behaviour appears when the perturbation magnitude decreases ($\epsilon = 0.5$ in first column of Figures \ref{ZepsilonFigure}-\ref{ZepsilonFigure2}, $\epsilon =0.3$ in second column and  $\epsilon =0.1$ in the third one).  Finally the choice of the level $u$ plays an important role in term of magnitude of obtained histograms (see the $y$-axis scale in Figures \ref{ZepsilonFigure}-\ref{ZepsilonFigure2}). This behaviour was already visible in the theoretical $\widetilde{h}_\epsilon$ function (see center and right panels of Figure \ref{fdeltafigure} for $u=1.5$ and $u=3$ respectively).

\section{Inference for perturbation}\label{sectionInference}

\subsection{Unbiased estimator of the perturbation}\label{estimatorPerturbation}

In this section we will focus on the case $\sigma_g=1$. Let  $u \neq 0$ being fixed. We introduce $\varepsilon: = \epsilon^2 \E[X^2]$.  Since $\E[X]=0$, it is clear that $\varepsilon$ quantifies the variability around zero of the considered perturbation and it can be useful to measure the discrepancy between the observed  excursion set $T \cap E_f(u)$ and the associated Gaussian one. \\

By using  \eqref{C2Tadditive} and then \eqref{eqGaussD}, we can rewrite
\begin{align*}
\E[C^{/T}_2(f,u)] &= C^*_2(g,u) +  \varepsilon  \frac{\pi}{\lambda} C^*_0(g,u) + O(\varepsilon^{3/2})=\Psi(u)+ \varepsilon \frac{u}{2 \sqrt{2 \pi} } \rme^{\frac{-u^2}{2}}  + O(\varepsilon^{3/2}).
\end{align*}
It appears then clearly that $\varepsilon$ has the same order of magnitude than
\begin{align}\label{epsilontilde}
{\varepsilon}_u := \frac{2 \sqrt{2 \pi}  \, \rme^{\frac{u^2}{2}}}{u}(C^*_2(f,u) - \Psi(u)).
\end{align}
This means that    $\varepsilon$ in \eqref{epsilontilde}  can be estimated by using the LK curvature of order $2$, \emph{i.e.}, the area of the excursion set at a (chosen) level $u$.    Then,  ${\varepsilon}_u$  is completely empirically accessible by using  this sparse observation   because it does not depend on the (unknown) second spectral moment $\lambda$ of the Gaussian field.  In Proposition \ref{epsilonProposition}  below, we   present a consistent estimator based on the  observation $T \cap E_f(u)$ for  the perturbation error ${\varepsilon}_u$.

\begin{Proposition}\label{epsilonProposition}
Let  $f(t) = g(t) + \epsilon \, X$,   $t\in \R^2$ as in Definition \ref{MODELadditive}. Let  $u \neq 0$ being fixed.
Let consider the empirical counterpart of ${\varepsilon}_u$ in  \eqref{epsilontilde},  i.e.,
\begin{align}\label{hateps}
\widehat{{\varepsilon}}_u :=   \frac{2 \sqrt{2 \pi} \rme^{\frac{u^2}{2}}}{u} \left(\widehat{C}_{2,T}(f,u) - \Psi(u)\right),
\end{align}
with $\widehat{C}_{2,T}(f,u)$ as in  \eqref{eq:UestC2}. Then, it holds that
\begin{description}
  \item[(\emph{i})]    $\widehat{{\varepsilon}}_u$ is an  unbiased estimator for ${\varepsilon}_u$, \vspace{-0.1cm}
   \item[(\emph{ii})] $|T^{(N)}|^{1/2}  (\widehat{{\varepsilon}}_u - {\varepsilon}_u)  \overset{d}{\underset{N\to\infty}{\longrightarrow}} N(0, \sigma^2_{{\varepsilon}_u})$, with $\sigma^2_{{\varepsilon}_u} = \,8 \pi  \,\frac{\rme^{u^2}}{u^2}\,v(u)$   for   $v(u)$ as in \eqref{vu}    with $\sigma_g=1$.
\end{description}
\end{Proposition}

\begin{proof}[Proof of Proposition \ref{epsilonProposition}]
Since $\widehat{C}_{2,T}(f,u)$ is an unbiased estimator of $C^*_2(f,u)$,  one can easily  see that $\E[\widehat{{\varepsilon}}_u - {\varepsilon}_u]=0$.   Furthermore,   using the fact that  $|T^{(N)}|^{1/2}  (\widehat{{\varepsilon}}_u - {\varepsilon}_u) =   \frac{\sqrt{8 \pi} \rme^{\frac{u^2}{2}}}{u} Y^{\epsilon_N}_{T^{(N)}}(u)$, from Theorem \ref{theoremThetaAsym}  we get the result.
\end{proof}

\begin{Remark}\label{remark0}\rm
If $u=0$, assuming that $\mathbb E[X^3]\ne 0$ and the fourth moment of $X$ is finite, then by Taylor developing the function $\mathbb E[C^{/T}_2(f,0)]$ up to the order $3$ (see Proposition \ref{LKadditive}) we easily get an unbiased and asymptotically normal estimator for $\epsilon^3 \mathbb E[X^3]$, similar to the r.h.s. of (\ref{hateps}).
\end{Remark}

\subsection{Numerical illustrations}

In this section we provide an illustration of the inference procedure for the perturbation $\varepsilon : = \epsilon^2\, \E[X^2]$  proposed in Section \ref{estimatorPerturbation} above. The considered perturbed model and the associated parameters are gathered in  Table \ref{parametersettingVarespsilon1}. By using this framework,  in Figure \ref{varepsilonFigure1} one can appreciate the finite sample performance of the   inference procedure  proposed  in  Section \ref{estimatorPerturbation} above, for several values of perturbation $\epsilon$ and for two levels $u$ ($u=1.5$ in center    panel and  $u=3$ in right one).

{\scriptsize
\begin{table}[H]
\centering
\begin{tabular}{c|c|c|c|c|c}
       \hline \hline
       {\Large\strut}   Level  $u$ & $X$ &   Chosen    $\epsilon$     &   $\varepsilon : = \epsilon^2\, \E[X^2]$    & average of estimated $\widehat{{\varepsilon}}_u$    &     Figure \ref{varepsilonFigure1}     \\
             &  &        &     & \tiny on 100 Montecarlo Simulations  &      \\ \hline
            {\large\strut}   \multirow{5}{*}{$1.5$} &  \multirow{5}{*}{\begin{tabular}{ c }
                                                                         Skellam  \\
                                                                        $\mu_1=\mu_2=1$ \\
                                                                       \end{tabular}}
                                      &0.1  &   0.02  &  0.023      &   \multirow{5}{*}{first panel}      \\ \cline{3-5}
             {\large\strut}      &     & 0.2 &   0.08   &  0.085   &       \\ \cline{3-5}
               {\large\strut}     &   & 0.3   &  0.18    & 0.182      &       \\ \cline{3-5}
               {\large\strut}     &   & 0.4   &  0.32    &   0.324      &       \\ \cline{3-5}
              {\large\strut}      & &  0.5  & 0.50   &    0.492  &     \\   \hline \hline
           {\large\strut}  \multirow{5}{*}{$3$} &  \multirow{5}{*}{\begin{tabular}{ c }
                                                                         Skellam  \\
                                                                        $\mu_1=\mu_2=1$ \\
                                                                       \end{tabular}}
                              &  0.1   &  0.02  &  0.033     &   \multirow{5}{*}{second  panel}     \\ \cline{3-5}
             {\large\strut}  &  & 0.2   &  0.08     & 0.072       &        \\ \cline{3-5}
           {\large\strut}    &  & 0.3 & 0.18   &   0.158     &        \\ \cline{3-5}
         {\large\strut}    &    & 0.4 &  0.32    &    0.345    &        \\ \cline{3-5}
         {\large\strut}    &    & 0.5 &  0.50  &   0.510       &       \\   \hline \hline
        \end{tabular} \vspace{0.3cm}
\caption{Parameters setting associated to Figure \ref{varepsilonFigure1}. Furthermore we consider $|T| = 1024^2$,  $\mu=0$, $\sigma_g=1$  and  $\rho(t)= e^{-\kappa^2 \parallel t \parallel^2}$,   for $\kappa = 100/2^{10}$, \emph{i.e.}, $\lambda= 0.019$. }\label{parametersettingVarespsilon1}
\end{table}}

\begin{figure}[H]
\hspace{-0.5cm}
\includegraphics[width=4.4cm, height=4.5cm]{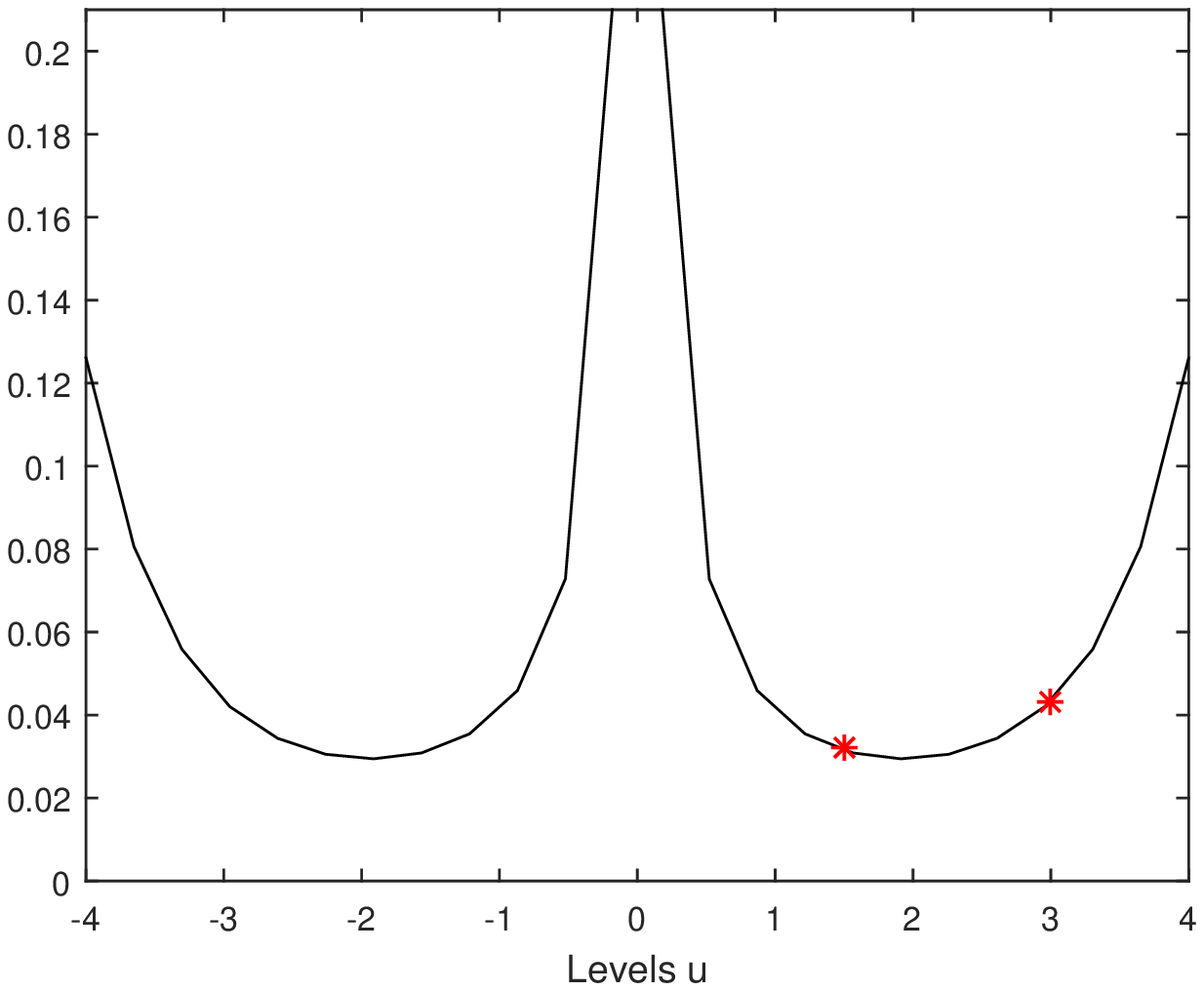}
\includegraphics[width=4.4cm, height=4.5cm]{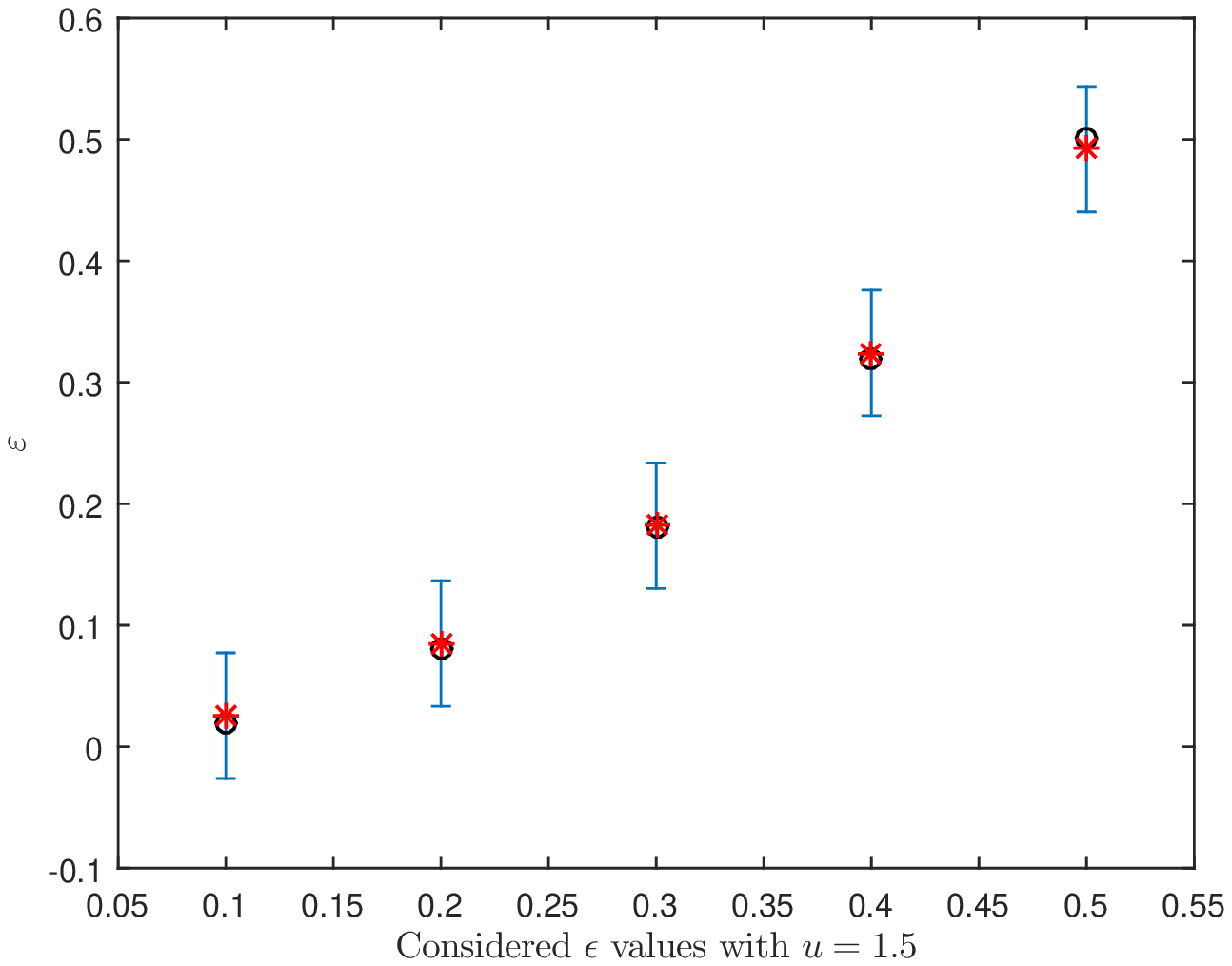}
\includegraphics[width=4.4cm,height=4.5cm]{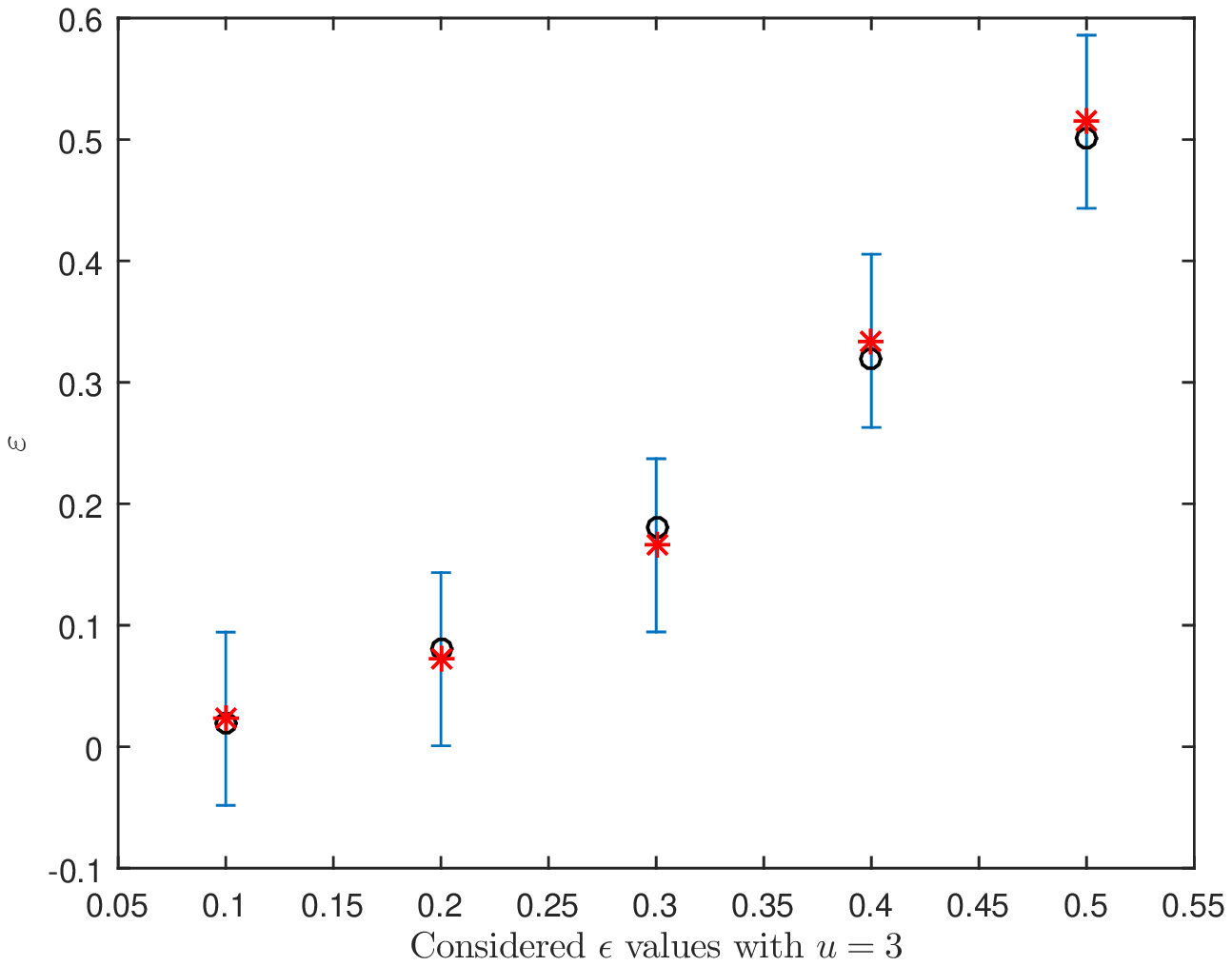}
    \vspace{-0.23cm}
   \caption{Left panel: theoretical $u \mapsto \sigma_{{\varepsilon}_u} |T^{(N)}|^{-1/2}$ in Proposition \ref{epsilonProposition}.  Center  and right panels: Theoretical  $\varepsilon : = \epsilon^2 \E[X^2]$  (black circles) and obtained average of  $\widehat{{\varepsilon}}_u$ on  100 Montecarlo independent    simulations (red stars).  Theoretical confidence intervals at level $0.95$ prescribed by Proposition \ref{epsilonProposition} are also displayed.       The chosen parameters setting is gathered  in Table \ref{parametersettingVarespsilon1}. Here $u=1.5$ (center  panel) and  $u=3$ (right panel). }\label{varepsilonFigure1}
\end{figure}

 Unsurprisingly, we remark that the variability of the estimation is related on the choice of level $u$.  The  asymptotic standard deviation  function  $u \mapsto \sigma_{{\varepsilon}_u}  |T^{(N)}|^{-1/2}$   in the left panel of Figure \ref{varepsilonFigure1}  allows us to identify some choices of levels $u$ where the variance is minimum.   Indeed,   for large values of $|u|$, less observations are available than for intermediate values of $|u|$.  This aspect can be appreciated by observing the larger confidence intervals in the  case $u=3$. For $u=0$, the variance  $\sigma_{{\varepsilon}_u}^2$ diverges (see the left panel of Figure \ref{varepsilonFigure1}) implying that this inference procedure will be not robust for $u \approx 0$ (see Remark \ref{remark0}).

\paragraph{Acknowledgements} The authors would like to thank Domenico Marinucci for insightful comments. The research of MR has been supported by the Fondation Sciences Math\'ematiques de Paris, and is currently supported by the ANR-17-CE40-0008 project \emph{Unirandom} and by the PRA 2018 49 project at the University of Pisa.

\appendix

\section{Uniform rates of convergence for CLTs of  sojourn times of stationary Gaussian fields}\label{appS}

The following lemma is a refinement of a result in \citet{pham}. Therefore, we keep the notations introduced therein.

\begin{Lemma}\label{lem_app}
Let $\lbrace X(t), t\in \mathbb R^d\rbrace$ be a stationary centered Gaussian field with unit variance and covariance function $\rho \in L^1(\mathbb R^d)$. For $T>0$ and $u\in \R$, we define $S_T(u)$ to be
\[
S_T(u) := \int_{[0, T]^d} 1_{(X(t)\ge u)}\,dt
\]
the excursion volume above level $u$. Under the hypothesis that
\begin{equation}\label{condPham}
\int_ {\mathbb R^d \setminus [-a,a]^d} |\rho(t)|\,dt = O\left ( 1/\log a \right ),\quad a\to +\infty,
\end{equation}
we have, as $T\to +\infty$,
\[
d_W \left ( \frac{S_T(u) - T^d\Phi(u)}{\sqrt{T^d}}, \mathcal N(0,\sigma^2(u))\right) = O\left (1/(\log T)^{1/12} \right),
\]
where the constant involved in the $O$-notation only depends on the field $\lbrace X(t), t\in \mathbb R^d\rbrace$ and
$$ 0<\sigma^2(u):=\sum_{n=1}^{+\infty} \frac{\phi^2(u) H^2_{n-1}(u)}{n!} \int_{\mathbb R^d} \rho^n(t)\,dt < +\infty,$$
$\phi$ being the density function of the standard Gaussian and $\Phi$ the tail of its distribution.
\end{Lemma}

\begin{Remark}
Actually, Theorem 2 in \citet{pham} ensures that, as $|T|\to +\infty$,
\begin{equation}\label{ineqPham}
d_W \left (  \frac{S_T(u) - T^d\Phi(u)}{\sqrt{T^d}}, \mathcal N(0,\sigma^2(u))\right) = O\left (1/(\log T)^{1/4}\right ),
\end{equation}
where the constant involved in the O-notation depends on the field and the level. By adapting the proof, we provide a \emph{uniform} rate of convergence w.r.t. the level $u$ in \eqref{ineqPham}.
\end{Remark}

\begin{proof}[Proof of Lemma \ref{lem_app}]
We will use the following estimate (see e.g. \cite[(30)]{hille1926} and \cite[Proposition 3]{Imkeller} ):  for every $n\in \mathbb N$ and $x\in \mathbb R$ we have
\begin{equation}\label{hille}
\exp(-x^2/4) |H_n(x)| \le K \,\sqrt{n!}\,n^{-1/12},
\end{equation}
where $K>0$ is an absolute constant.
We can write
\begin{equation}
\begin{split}
d_W\left (  \frac{S_T(u) - T^d\Phi(u)}{\sqrt{T^d}}, \mathcal N(0,\sigma^2(u) )     \right )
& \le d_W \left (\frac{S_T(u) - T^d\Phi(u)}{\sqrt{T^d}}, \frac{S_{T,N_T}(u) - \mathbb E[S_{T,N_T}(u)]}{\sqrt{T^d}}\right ) \cr
& + d_W \left ( \frac{S_{T,N_T}(u) - \mathbb E[S_{T,N_T}(u)]}{\sqrt{T^d}}, \mathcal N(0, \sigma^2_{N_T}(u) )     \right ) \cr
& + d_W \left ( \mathcal N(0, \sigma^2(u) )  , \mathcal N(0, \sigma^2_{N_T}(u) )     \right ) =: d_1 + d_2 + d_3,
\end{split}
\end{equation}
where $S_{T,N_T}$ is the truncation of $S_T$ at position $N_T$ in the Wiener chaos expansion ($N_T$ will be chosen later on).
For $d_1$ we have, due to \eqref{hille},
\begin{equation}\label{bound1}
\begin{split}
d_1 \le & \sqrt{\sum_{n=N_T +1}^{+\infty} \frac{\phi^2(u) H_{n-1}^2(u)}{n! T^d} \int_{[-T,T]^d} \rho^n(t) \prod_{j=1}^d (T- |t_j|)\,dt } \cr
\le & K \sqrt{\phi(u)} \sqrt{\int_{\mathbb R^d} |\rho(t)|\,dt} \sqrt{\sum_{n=N_T +1}^{+\infty} \frac{1}{n (n-1)^{1/6}}}  \cr
= & O(N_T^{-1/12}),
\end{split}
\end{equation}
where the constant involved in the O-notation only depends on the field.

Note that due to \eqref{hille} we can give upper bounds for $d_2$ and $d_3$  (being inspired by the proof of Theorem 2 in \citet{pham})  independently of $u$
\begin{equation}\label{bound2e3}
d_2 = O(3^{N_T}/\sqrt{T^d}), \qquad d_3 =O_{T\to +\infty}(N_T^{-1/12} + T^{-1/4} + (\log T)^{-1/2}).
\end{equation}

Summing up the bounds for $d_1, d_2, d_3$ in \eqref{bound1} and \eqref{bound2e3}, and choosing $N_T = (\log T)/4$ we have
$$
d_1 + d_2 + d_3 = O_{T\to +\infty}((\log T)^{-1/12})
$$
that concludes the proof.
\end{proof}

\bibliographystyle{apalike}
\bibliography{biblioLKC}

\begin{thebibliography}{}

\bibitem[Adler et~al., 2010]{adler_samorodnitsky_taylor_2010}
Adler, R.~J., Samorodnitsky, G., and Taylor, J.~E. (2010).
\newblock Excursion sets of three classes of stable random fields.
\newblock {\em Advances in Applied Probability}, 42(2):293--318.

\bibitem[Adler and Taylor, 2007]{AT07}
Adler, R.~J. and Taylor, J.~E. (2007).
\newblock {\em Random fields and geometry}.
\newblock Springer Monographs in Mathematics. Springer, New York.

\bibitem[Adler and Taylor, 2011]{AT11}
Adler, R.~J. and Taylor, J.~E. (2011).
\newblock {\em Topological complexity of smooth random functions}, volume 2019
  of {\em Lecture Notes in Mathematics}.
\newblock Springer, Heidelberg.
\newblock Lectures from the 39th Probability Summer School held in Saint-Flour,
  2009, \'Ecole d'\'Et\'e de Probabilit\'es de Saint-Flour. [Saint-Flour
  Probability Summer School].

\bibitem[Beliaev et~al., 2018]{nodal}
Beliaev, D., McAuley, M., and Muirhead, S. (2018).
\newblock On the number of excursion sets of planar {G}aussian fields.
\newblock {\em ArXiv e-prints 1807.10209}.

\bibitem[{Berzin}, 2018]{Berzin}
{Berzin}, C. (2018).
\newblock {Estimation of Local Anisotropy Based on Level Sets}.
\newblock {\em ArXiv e-prints:1801.03760}.

\bibitem[Beuman et~al., 2012]{BTV}
Beuman, T.~H., Turner, A.~M., and Vitelli, V. (2012).
\newblock Stochastic geometry and topology of non-{G}aussian fields.
\newblock {\em Proceedings of the National Academy of Sciences},
  109(49):19943--19948.

\bibitem[Bierm\'e and Desolneux, 2016]{BD16}
Bierm\'e, H. and Desolneux, A. (2016).
\newblock On the perimeter of excursion sets of shot noise random fields.
\newblock {\em The Annals of Probability}, 44(1):521--543.

\bibitem[Bierm{\'e} et~al., 2019]{BDDE}
Bierm{\'e}, H., Di~Bernardino, E., Duval, C., and Estrade, A. (2019).
\newblock Lipschitz-{K}illing curvatures of excursion sets for two-dimensional
  random fields.
\newblock {\em Electronic Journal of Statistics}, 13(1):536--581.

\bibitem[Bron and Jeulin, 2011]{ImageAnalStereol751}
Bron, F. and Jeulin, D. (2011).
\newblock Modelling a food microstructure by random sets.
\newblock {\em Image Analysis \& Stereology}, 23(1).

\bibitem[Bulinski et~al., 2012]{bulinski2012central}
Bulinski, A., Spodarev, E., and Timmermann, F. (2012).
\newblock Central limit theorems for the excursion set volumes of weakly
  dependent random fields.
\newblock {\em Bernoulli}, 18(1):100--118.

\bibitem[Burgess, 1999]{Burgess}
Burgess, A.~E. (1999).
\newblock Mammographic structure: Data preparation and spatial statistics
  analysis.
\newblock {\em Medical Imaging'99. International Society for Optics and
  Photonics}, pages 642--653.

\bibitem[Caba{\~n}a, 1987]{cabana1987affine}
Caba{\~n}a, E.~M. (1987).
\newblock Affine processes: a test of isotropy based on level sets.
\newblock {\em SIAM Journal on Applied Mathematics}, 47(4):886--891.

\bibitem[Collaboration, 2016]{Yabi1}
Collaboration, P. (2016).
\newblock Planck 2015 results. {XVI}. {I}sotropy and statistics of the {CMB}.
\newblock {\em Astronomy and {A}strophysics}, 594(A16).

\bibitem[Estrade and Le\'{o}n, 2016]{annejose}
Estrade, A. and Le\'{o}n, J.~R. (2016).
\newblock A central limit theorem for the {E}uler characteristic of a
  {G}aussian excursion set.
\newblock {\em The Annals of Probability}, 44(6):3849--3878.

\bibitem[Fantaye et~al., 2015]{Yabi2}
Fantaye, Y., Marinucci, D., Hansen, F., and Maino, D. (2015).
\newblock Applications of the {G}aussian kinematic formula to {C}{M}{B} data
  analysis.
\newblock {\em Phys. Rev. D}, 91:063501.

\bibitem[Hassan and Hijazi, 2010]{Hassan}
Hassan, M.~Y. and Hijazi, R.~H. (2010).
\newblock A {B}imodal {E}xponential {P}ower {D}istribution.
\newblock {\em Pakistan Journal of Statistics}, 26(2):379 -- 396.

\bibitem[Hikage and Matsubara, 2012]{Matsubara2}
Hikage, C. and Matsubara, T. (2012).
\newblock Limits on {S}econd-{O}rder {N}on-{G}aussianity from {M}inkowski
  {F}unctionals of {WMAP} {D}ata.
\newblock {\em Mon. Not. R. Astron. Soc.}, 425:2187--2196.

\bibitem[Hille, 1926]{hille1926}
Hille, E. (1926).
\newblock A {C}lass of {R}eciprocal {F}unctions.
\newblock {\em Annals of Mathematics}, 27(4):427 -- 464.

\bibitem[Imkeller et~al., 1995]{Imkeller}
Imkeller, P., P\'{e}rez-Abreu, V., and Vives, J. (1995).
\newblock Chaos expansions of double intersection local time of {B}rownian
  motion in {${\bf R}^d$} and renormalization.
\newblock {\em Stochastic Process. Appl.}, 56(1):1--34.

\bibitem[Kratz and Le\'{o}n, 2010]{KratzLeon}
Kratz, M. and Le\'{o}n, J. (2010).
\newblock Level curves crossings and applications for {G}aussian models.
\newblock {\em Extremes}, 13(3):315--351.

\bibitem[Kratz and Vadlamani, 2017]{KV16}
Kratz, M. and Vadlamani, S. (2017).
\newblock Central limit theorem for {L}ipschitz--{K}illing curvatures of
  excursion sets of {G}aussian random fields.
\newblock {\em Journal of Theoretical Probability}.

\bibitem[Lachi\`{e}ze-Rey, 2017]{Rapha}
Lachi\`{e}ze-Rey, R. (2017).
\newblock Shot-noise excursions and non-stabilizing {P}oisson functionals.

\bibitem[Marinucci and Peccati, 2011]{MPbook}
Marinucci, D. and Peccati, G. (2011).
\newblock {\em Random fields on the sphere}, volume 389 of {\em London
  Mathematical Society Lecture Note Series}.
\newblock Cambridge University Press, Cambridge.
\newblock Representation, limit theorems and cosmological applications.

\bibitem[Marinucci and Rossi, 2015]{maudom}
Marinucci, D. and Rossi, M. (2015).
\newblock Stein-{M}alliavin approximations for nonlinear functionals of random
  eigenfunctions on $\mathbb{S}^d$.
\newblock {\em Journal of Functional Analysis}, 268(8):2379--2420.

\bibitem[Matsubara, 2010]{Matsubara1}
Matsubara, T. (2010).
\newblock Analytic minkowski functionals of the cosmic microwave background:
  Second-order non-gaussianity with bispectrum and trispectrum.
\newblock {\em Phys. Rev. D}, 81:083505.

\bibitem[M{\"u}ller, 2017]{Mu16}
M{\"u}ller, D. (2017).
\newblock A central limit theorem for {L}ipschitz--{K}illing curvatures of
  {G}aussian excursions.
\newblock {\em Journal of Mathematical Analysis and Applications},
  452(2):1040--1081.

\bibitem[Pham, 2013]{pham}
Pham, V.-H. (2013).
\newblock On the rate of convergence for central limit theorems of sojourn
  times of {G}aussian fields.
\newblock {\em Stochastic Processes and their Applications}, 123(6):427 -- 464.

\bibitem[Roberts and Teubner, 1995]{Roberts1995}
Roberts, A. and Teubner, M. (1995).
\newblock Transport properties of heterogeneous materials derived from
  {G}aussian random fields: {B}ounds and simulation.
\newblock {\em Phys Rev E Stat Phys Plasmas Fluids Relat Interdiscip Topics},
  51(5):4141 -- 4154.

\bibitem[Roberts and Torquato, 1999]{Roberts1999}
Roberts, A. and Torquato, S. (1999).
\newblock Chord-distribution functions of three-dimensional random media:
  approximate first-passage times of {G}aussian processes.
\newblock {\em Phys Rev E Stat Phys Plasmas Fluids Relat Interdiscip Topics},
  59(5):4953 -- 4963.

\bibitem[Schneider and Weil, 2008]{Weil08}
Schneider, R. and Weil, W. (2008).
\newblock {\em Stochastic and integral geometry}.
\newblock Probability and its Applications. Springer-Verlag, Berlin.

\bibitem[Spodarev, 2014]{Spodarev13}
Spodarev, E. (2014).
\newblock Limit theorems for excursion sets of stationary random fields.
\newblock In {\em Modern stochastics and applications}, volume~90 of {\em
  Springer Optim. Appl.}, pages 221--241. Springer, Cham.

\bibitem[Th\"ale, 2008]{Thale08}
Th\"ale, C. (2008).
\newblock 50 years sets with positive reach - a survey.
\newblock {\em Surveys in Mathematics and its Applications}, 3:123--165.

\bibitem[Worsley, 1997]{Worsley}
Worsley, K.~J. (1997).
\newblock The geometry of random images.
\newblock {\em Chance}, (9):27--40.

\end{thebibliography}

\end{document}